%% file: The_second_homology_of_the_homological_Goldman_Lie_algebra.tex
\documentclass[12pt]{article}
\usepackage{amsmath}
\usepackage{amssymb}
\usepackage{amsthm}
\usepackage{color}

\newtheorem{Thm}{Theorem}
\newtheorem{Lem}[Thm]{Lemma}
\newtheorem{Prop}[Thm]{Proposition}
\newtheorem{Cor}[Thm]{Corollary}

\def\mbN{\mathbb{N}}
\def\mbZ{\mathbb{Z}}
\def\mbQ{\mathbb{Q}}
\def\mbR{\mathbb{R}}

\title{The second homology group of the homological Goldman Lie algebra}

\author{Kazuki Toda}

\begin{document}
\maketitle
\begin{abstract}
We determine the second homology group of the homological Goldman Lie algebra for an oriented surface. 
\end{abstract}
\section{Introduction}

    By a {\it surface}, we mean an oriented two-dimensional smooth manifold possibly with boundary. 
    The first homology group of a surface and its intersection form reflect the topological structure of the surface. 
    For example, they have information about the genus and the boundary components of the surface. 
    To study them in detail, we consider a Lie algebra coming from them. 
    Goldman introduced the Lie algebra for study of the moduli space of  $GL_1(\mbR )$-flat bundles over the surface. 
    We define the Lie algebra in more general setting. 
    
    Let $H$ be a $\mbZ$-module, which is not necessary finitely generated, and $\langle -,-\rangle :H\times H\to \mbZ$, $(x,y)\mapsto \langle x,y\rangle$, an alternating $\mbZ$-bilinear form. 
    For example, we consider that $H$ is the first homology group of a surface and $\langle -,-\rangle$ is its intersection form. 
    We define $\mbZ$-linear map $\mu :H\to {\rm Hom}_{\mbZ}(H,\mbZ )$ by $\mu (x)(y)=\langle x,y\rangle$ for $x,y\in H$. 
    Denote by $\mbQ [H]$ the $\mbQ$-vector space with basis the set $H$; 
    \[
        \mbQ [H]=\left\{ \sum _{i=1}^m c_i[x_i] | n\in \mbN ,c_i \in \mbQ ,x_i \in H \right\} , 
    \]
    where $[-]:H\to \mbQ [H]$ is the embedding as basis. 
    We remark $c[x]\neq [c x]\in \mbQ [H]$ for $c\neq 1$. 
    We define $\mbQ$-linear map $[-,-]:\mbQ [H]\times \mbQ [H]\to \mbQ [H]$ by $[[x],[y]]=\langle x,y\rangle [x+y]$ for $x,y\in H$. 
    It is easy to see that this bilinear map is skew and satisfies the Jacobi identity \cite[p.295-p.297]{G}. 
    The Lie algebra $(\mbQ [H],[-,-])$ is called the {\it homological Goldman Lie algebra of $(H,\langle -,-\rangle )$}. 
    
    Our purpose is to study the algebraic structure of the homological Goldman Lie algebra. 
    In the previous paper \cite{T}, we determined all the ideals of the homological Goldman Lie algebra. 
    In particular, the derived Lie subalgebra $\mbQ [H]^{(1)}=[\mbQ [H],\mbQ [H]]$ of $\mbQ [H]$ is the $\mbQ$-vector subspace $\mbQ [H\setminus \ker \mu ]$ with the basis the set $H\setminus \ker \mu$, the center $\mathfrak{z}(\mbQ [H])$ of $\mbQ [H]$ is the $\mbQ$-vector subspace $\mbQ [\ker \mu ]$ with the basis the set $\ker \mu$, and the abelianization $\mbQ [H]^{\rm ab}=\mbQ [H]/\mbQ [H] ^{(1)}$ of $\mbQ [H]$ equals the center. 
    In other words, the first homology group $H_1(\mbQ [H])$ of $\mbQ [H]$ equals the center. 
    In the present paper, we determine the second homology group \cite{H-S} of the homological Goldman Lie algebra. 
    That is, our main theorem in this paper is the following. 
    \begin{Thm} \label{2nd homology thm}
        If $\langle -,-\rangle \neq 0$, we have isomorphisms
        \[
            H_2(\mbQ [H]^{(1)})\to \bigoplus _{z\in \ker \mu}\mbQ \otimes (H/\mbZ z) \mbox{ and } 
        \]
        \[
            H_2(\mbQ [H])\to (\wedge ^2\mbQ [\ker \mu ]) \oplus H_2(\mbQ [H]^{(1)}) . 
        \]
    \end{Thm}
    Goldman introduced a more geometric Lie algebra as follows. 
    Let $\Sigma$ be a surface. 
    Denote by $[S^1,\Sigma ]$ the set of free homotopy classes of free loops on $\Sigma$. 
    For two free loops $\alpha$ and $\beta$ on $\Sigma$ in general position, we define $[\alpha ,\beta ]=\sum _{p\in \alpha \cap \beta}\varepsilon (p;\alpha ,\beta )[\alpha \cdot _p \beta ]$, where $\varepsilon (p;\alpha ,\beta )$ is the local intersection number of $\alpha$ and $\beta$ at $p$, and $\alpha \cdot _p \beta$ is the free homotopy class of the product in the fundamental group $\pi _1(\Sigma ,p)$ with base point $p$. 
    The bracket induces a well-defined binary operation in $\mbQ [[S^1,\Sigma ]]$, the $\mbQ$-vector space with basis the set of homotopy classes of free loops, and the operation is skew and satisfies the Jacobi identity \cite{G}. 
    The Lie algebra $(\mbQ [[S^1,\Sigma ]],[-,-])$ is called the {\it Goldman Lie algebra of $\Sigma$}. 
    Goldman used the Lie algebra $\mbQ [[S^1,\Sigma ]]$ for study of the space of representations on the fundamental group $\pi _1(\Sigma )$ of the surface $\Sigma$. 
    The natural projection $[S^1,\Sigma ]\to H_1(\Sigma ;\mbZ )$ induces the surjective Lie algebra homomorphism $\mbQ [[S^1,\Sigma ]]\to \mbQ [H_1(\Sigma ;\mbZ )]$. 
    Hence we can have some information of the Goldman Lie algebra through the homomorphism $\mbQ [[S^1,\Sigma ]]\to \mbQ [H_1(\Sigma ;\mbZ )]$. 
    In fact, as will be shown in Theorem \ref{Goldman 2nd hom}, we have that if $\Sigma$ is connected, then the composition $H_2(\mbQ [[S^1,\Sigma ]])\to H_2(\mbQ [H_1(\Sigma ;\mbZ )]) \to H_2(\mbQ [H_1(\Sigma ;\mbZ )]^{(1)})$ is surjective, where the map $H_2(\mbQ [H_1(\Sigma ;\mbZ )]) \to H_2(\mbQ [H_1(\Sigma ;\mbZ )]^{(1)})$ is the induced map by the projection $\varpi ^{(1)}:\mbQ [H_1(\Sigma ;\mbZ )]=\mbQ [\ker \mu ]\oplus \mbQ [H_1(\Sigma ;\mbZ )\setminus \ker \mu] \to \mbQ [H_1(\Sigma ;\mbZ )\setminus \ker \mu ]$. 
    
    Define the $\mbQ$-linear map $K:\mbQ [H]\to \mbQ \bigotimes _{\mbZ}H$ by $K([x])=1\otimes x$ for $x\in H$. 
    The $\mbQ$-vector subspace $\mathfrak{g}_K:=\ker K$ is a Lie subalgebra of $\mbQ [H]$. 
    In Theorem \ref{G_K 2nd hom}, we prove that if $\langle -,-\rangle \neq 0$, then the composition $H_2(\mathfrak{g}_K)\to H_2(\mbQ [H])\to H_2(\mbQ [H]^{(1)})$ is surjective. 
    The Lie subalgebra $\mathfrak{g}_K$ is related to the Kontsevich's "commutative" \cite{Kon} as follows. 
    Assume that $H$ has a symplectic basis $\{ x_i,y_i\} _{i=1}^g$. 
    Then $\mbQ \otimes H$ has a symplectic basis $\{ 1\otimes x_i,1\otimes y_i\} _{i=1}^g$. 
    The symmetric algebra $\prod _{m=0}^{\infty} {\rm Sym}^m(\mbQ \otimes H)$ equips a Poisson bracket defined by
    \[
        [f,h] = \sum _{i=1}^g\left( \frac{\partial f}{\partial (1\otimes x_i)}\frac{\partial h}{\partial (1\otimes y_i)}-\frac{\partial h}{\partial (1\otimes x_i)}\frac{\partial f}{\partial (1\otimes y_i)}\right) . 
    \]
    where $f,h\in \prod _{m=0}^{\infty} {\rm Sym} ^m(\mbQ \otimes H)$ is regarded as formal power series in variables $\{ 1\otimes x_i,1\otimes y_i\} _{i=1}^g$. 
    The constant terms $\mbQ \subset \prod _{m=0}^{\infty} {\rm Sym} ^m(\mbQ \otimes H)$ is included in the center of $\prod _{m=0}^{\infty} {\rm Sym} ^m(\mbQ \otimes H)$. 
    Hence $\prod _{m=0}^{\infty} {\rm Sym} ^m(\mbQ \otimes H)/\mbQ$ has a natural Lie bracket. 
    The composite of the inclusion and the quotient projection $\prod _{m=1}^{\infty} {\rm Sym} ^m(\mbQ \otimes H)\to \prod _{m=0}^{\infty} {\rm Sym} ^m(\mbQ \otimes H)/\mbQ$ is a $\mbQ$-linear isomorphism. 
    Then $\prod _{m=1}^{\infty} {\rm Sym} ^m(\mbQ \otimes H)$ also has a Lie bracket. 
    The $\mbQ$-linear map $\mbQ [H] \to \prod _{m=0}^{\infty} {\rm Sym} ^m(\mbQ \otimes H)$, $[x]\mapsto {\rm exp}(1\otimes x)$, is a Lie algebra homomorphism. 
    By composition, we have a Lie algebra homomorphism $\mbQ [H]\to \prod _{m=1}^{\infty} {\rm Sym} ^m(\mbQ \otimes H)$. 
    We denote $\mathfrak{c}_g:=\prod _{m=2}^{\infty} {\rm Sym}^m (\mbQ \otimes H)$, which is a Lie subalgebra of $\prod _{m=1}^{\infty} {\rm Sym}^m (\mbQ \otimes H)$ and the degree completion of Kontsevich's "commutative". 
    The Lie subalgebra $\mathfrak{g}_K$ is exactly the inverse image of $\mathfrak{c}_g$ by the Lie algebra homomorphism stated above. 
    The Lie algebra $\mathfrak{c}_g$ equals the Lie algebra of formal Hamiltonian vector fields $\mathfrak{ham}_{2 g}$, which has been studied in the context of symplectic foliation \cite{GKF} \cite{M} \cite{K-M}. 
    It would be interesting if we could describe the relation between the second cohomology group of $\mathfrak{ham}_{2 g}$ and that of our $\mathfrak{g}_K$. 
    
    In Theorem \ref{3rd coh}, we construct an explicit nontrivial cohomology class in $H^3(\mbQ [H])$ if $\langle -,-\rangle \neq 0$. 
    In particular, we have $H^3(\mbQ [H])\neq 0$.

\section{Preparation}
\subsection{Derived algebra of some subalgebra}
Let $S$ be a subset of $H$. 
Then we define $S^{(1)}:=\{ u+v\in H \mid u,v\in S ,\langle u,v\rangle \neq 0\}$.  
For example, we have 
\[
H^{(1)}=H\setminus {\rm ker}\mu . 
\]
In fact, assume $x\in H^{(1)}$. 
Then there exist $u$ and $v\in H$ with $\langle u,v\rangle \neq 0$ and $x=u+v$. 
Since $\langle x,v\rangle =\langle u,v\rangle \neq 0$, we have $x\in H\setminus {\rm ker}\mu$ . 
Conversely, assume $x\in H\setminus {\rm ker}\mu$. 
Then there exists $y\in H$ with $\langle x,y\rangle \neq 0$. 
Since $\langle x-y,y\rangle =\langle x,y\rangle \neq 0$ and $x=(x-y)+y$, we have $x\in H^{(1)}$. 
\begin{Prop}
$\mbQ [S]$ is a subalgebra of $\mbQ [H]$ if and only if $S^{(1)}\subset S$. 
Then, $(\mbQ [S])^{(1)}=[\mbQ [S],\mbQ [S]]=\mbQ [S^{(1)}]$. 
\end{Prop}
\begin{proof}
Assume $\mbQ [S]$ is a subalgebra of $\mbQ [H]$. 
For $x\in S^{(1)}$, there exist $u$ and $v\in S$ with $x=u+v$ and $\langle u,v\rangle \neq 0$. 
We have $[x]=\frac{1}{\langle u,v\rangle}[[u],[v]]\in [\mbQ [S],\mbQ [S]]\subset \mbQ [S]$. 
Hence, we have $x\in S$. 
This proves $S^{(1)}\subset S$ and $\mbQ [S^{(1)}]\subset [\mbQ [S],\mbQ [S]]$. 

Assume $S^{(1)}\subset S$. 
It is trivial $\mbQ [S]$ is a subspace of $\mbQ [H]$ as $\mbQ$-vector space. 
For $u$ and $v\in S$, we have 
\[
[[u],[v]]
=
\left\{ 
\begin{matrix}
0, & \mbox{if\ } \langle u,v\rangle =0, \\
\langle u,v\rangle [u+v], & \mbox{if\ } \langle u,v\rangle \neq 0 . 
\end{matrix}
\right. 
\]
In both cases we have $[[u],[v]]\in \mbQ [S^{(1)}]\subset \mbQ [S]$. 
Hence $\mbQ [S]$ is a subalgebra of $\mbQ [H]$. 
And this shows $[\mbQ [S],\mbQ [S]]\subset \mbQ [S^{(1)}]$. 
This completes the proof of the proposition. 
\end{proof}
The center $\mathfrak{z}(\mbQ [H])$ of $\mbQ [H]$ is $\mbQ [{\rm ker}\mu ]$. 
In fact, assume $u\in {\rm ker}\mu$. 
Then we have $[u] \in \mathfrak{z}(\mbQ [H])$, since $[[u],[v]]=\langle u,v\rangle [u+v]=0$ for $v\in H$. 
Conversely, assume $u \in H\setminus {\rm ker}\mu =H^{(1)}$. 
Then there exists $v\in H$ with $\langle u,v\rangle \neq 0$. 
Then we have $[u] \in \mbQ [H]\setminus \mathfrak{z}(\mbQ [H])$ since $[[u],[v]]=\langle u,v\rangle [u+v]\neq 0$. 

The inclusion $\iota ^{(1)}:\mbQ [H^{(1)}]\rightarrow \mbQ [H]$ induces the inclusion homomorphism $C_p(\iota ^{(1)}):C_p(\mbQ [H^{(1)}])\rightarrow C_p(\mbQ [H])$ and the restriction $C^p(\iota ^{(1)}):C^p(\mbQ [H])\rightarrow C^p(\mbQ [H^{(1)}])$. 
On the other hand, 
we can decompose $\mbQ [H]$ into the center and the derived subalgebra, $\mbQ [H]=\mbQ [{\rm ker}\mu ]\oplus \mbQ [H^{(1)}]$. 
The projection $\varpi ^{(1)}:\mbQ [H]\rightarrow \mbQ [H^{(1)}]$ of the decomposition induces the projection $C_p(\varpi ^{(1)}):C_p(\mbQ [H])\rightarrow C_p(\mbQ [H^{(1)}])$ and the zero extension $C^p(\varpi ^{(1)}) :C^p(\mbQ [H^{(1)}])\rightarrow C^p(\mbQ [H])$. 
The chain maps $C_*(\iota ^{(1)})$ and $C_*(\varpi ^{(1)})$ satisfy $C_*(\varpi ^{(1)})\circ C_*(\iota ^{(1)})=id$ since $\varpi ^{(1)}\circ \iota ^{(1)}=id$. 
By the same reason, the cochain maps $C^*(\iota ^{(1)})$ and $C^*(\varpi ^{(1)})$ satisfy $C^*(\iota ^{(1)})\circ C^*(\varpi ^{(1)})=id$. 
\subsection{Decomposition of (co)homology}
Let $S$ be a subset of $H$ with $S^{(1)}\subset S$. 
For $p>0$ and $z\in H$, denote by $C_p(\mbQ [S])_{(z)}$ a subspace of $C_p(\mbQ [S])$ generated by the set $\{ [u_1]\wedge \cdots \wedge [u_p]\mid u_1,\ldots ,u_p\in S,u_1+\cdots +u_p=z\}$. 
we define a subspace $C^p(\mbQ S)_{(z)}$ of $C^p(\mbQ S)$ by 
\[
C^p(\mbQ [S])_{(z)}:=\left\{ \omega \in C^p(\mbQ S) \mid 
\begin{matrix}
u_1,\cdots ,u_p\in S, u_1+\cdots +u_p \neq z \\
\Longrightarrow \omega ([u_1],\cdots ,[u_p])=0 
\end{matrix}
\right\} . 
\]
The subspace $C_*(\mbQ [S])_{(z)}$ is a subcomplex of the chain complex $C_*(\mbQ [S])$ and 
the subspace $C^*(\mbQ [S])_{(z)}$ is a subcomplex of the cochain complex $C^*(\mbQ S )$, that is, 
we have $\partial (C_p(\mbQ [S])_{(z)})\subset C_{p-1}(\mbQ [S])_{(z)}$ and $d(C^p(\mbQ S)_{(z)})\subset C^{p+1}(\mbQ S)_{(z)}$ for all $p>0$ because $u_1+\cdots +u_p$ define the degree of $\wedge \mbQ [S]$. 
We denote
$Z_p(\mbQ [S])_{(z)}:=Z_p(C_*(\mbQ [S])_{(z)})$, 
$B_p(\mbQ [S])_{(z)}:=B_p(C_*(\mbQ [S])_{(z)})$, 
$H_p(\mbQ [S])_{(z)}:=H_p(C_*(\mbQ [S])_{(z)})$, 
$Z^p(\mbQ [S])_{(z)}:=Z^p(C^*(\mbQ [S])_{(z)})$, 
$B^p(\mbQ [S])_{(z)}:=B^p(C^*(\mbQ [S])_{(z)})$ and
$H^p(\mbQ [S])_{(z)}:=H^p(C^*(\mbQ [S])_{(z)})$. 
Then we have 
$Z_p(\mbQ [S])\cong \prod _{z\in H}Z_p(\mbQ [S])_{(z)}$, 
$B_p(\mbQ [S])\cong \prod _{z\in H}B_p(\mbQ [S])_{(z)}$, 
$H_p(\mbQ [S])\cong \prod _{z\in H}H_p(\mbQ [S])_{(z)}$, 
$Z^p(\mbQ [S])\cong \prod _{z\in H}Z^p(\mbQ [S])_{(z)}$, 
$B^p(\mbQ [S])\cong \prod _{z\in H}B^p(\mbQ [S])_{(z)}$ and 
$H^p(\mbQ [S])\cong \prod _{z\in H}H^p(\mbQ [S])_{(z)}$. 

We call $H_p(\mbQ [S])_{(z)}$ or $H_p(\mbQ [S])_{(z)}$ an {\it inner component} if $z\in {\rm ker}\mu$, and an {\it outer component} if $z\in H^{(1)}$. 

    \section{Inner component}\label{sec:inner}
        Fix an element $z\in {\rm ker}\mu$ in this section. 
        Denote by $\hat{I}$ the $\mbQ$-algebra ideal in $\wedge \mbQ [H]$ generated by the set $\{ [u+v]\wedge [x]-[u]\wedge [x+v]-[v]\wedge [x+u]\mid u,v,x\in H\}$. 
        The generator system of $\hat{I}$  consists of homogeneous elements about degree $p\in \{ 0,1,2,\ldots \}$. 
        Then we have $\hat{I}=\bigoplus _{p=0}^{\infty}\hat{I}\cap C_p(\mbQ [H])$. 
        The generator system of $\hat{I}$ consists of homogeneous elements about degree $x\in H$. 
        Then we have $\hat{I}\cap C_p(\mbQ [H])=\bigoplus _{x\in H}\hat{I}\cap C_p(\mbQ [H])_{(x)}$. 
        Set $\hat{I}_{p,(x)}=\hat{I}\cap C_p(\mbQ [H])_{(x)}$ and $\hat{C}_p(\mbQ [H])_{(x)}=C_p(\mbQ [H])/\hat{I}_{p,(x)}$. 
        
        Dually, we define the subspace $\hat{C}^p(\mbQ [H])_{(z)}$ of $C^p(\mbQ [H])_{(z)}$ by the following condition; 
        $\omega \in \hat{C}^p(\mbQ [H])_{(z)}$ if and only if the map $H^p=H\times \cdots \times H\ni (u_1,\ldots ,u_{p-1})\mapsto \omega ([u_1], \ldots ,[u_{p-1}],[z-u_1-\cdots -u_{p-1}]) \in \mbQ$ is a $\mbZ$-$(p-1)$-linear map. 
        \begin{Prop}\label{comp zero}
            The composition $(\mbox{the quotient projection})\circ \partial _p :C_p(\mbQ [H])_{(z)}\to C_{p-1}(\mbQ [H])_{(z)}\to \hat{C}_{p-1}(\mbQ [H])_{(z)}$ is a zero map. 
        \end{Prop}
        \begin{proof}
            For $u_1,\ldots ,u_p\in H$ with $u_1+\cdots +u_p=z$, we have 
            \begin{eqnarray*}
                &&\partial _p([u_1]\wedge \cdots \wedge [u_p]) \\
                &=& \sum _{1\leq i<j\leq p-1}(-1)^{i+j}\langle u_i,u_j\rangle [u_i+u_j]\wedge [u_1]\wedge \cdots \wedge \hat{[u_i]}\wedge \cdots \wedge \hat{[u_j]}\wedge \cdots \wedge [u_p] \\
                &&+\sum _{k=1}^{p-1}(-1)^{i+p}\langle u_k,u_p\rangle [u_k+u_p]\wedge [u_1]\wedge \cdots \wedge \hat{[u_k]}\wedge \cdots \wedge [u_{p-1}] \\
                &\equiv & \sum _{1\leq i<j\leq p-1}(-1)^{i+j}\langle u_i,u_j\rangle \\
                &&\qquad \big( [u_i]\wedge [u_1]\wedge \cdots \wedge \hat{[u_i]}\wedge \cdots \wedge \hat{[u_j]}\wedge \cdots \wedge [u_{p-1}]\wedge [u_p+u_j] \\
                &&\qquad +[u_j]\wedge [u_1]\wedge \cdots \wedge \hat{[u_i]}\wedge \cdots \wedge \hat{[u_j]}\wedge \cdots \wedge [u_{p-1}]\wedge [u_p+u_i] \big) \\
                &&+\sum _{k=1}^{p-1}(-1)^{i+p}\langle u_k,z-\sum _{\ell =1}^{p-1}u_{\ell}\rangle [u_k+u_p]\wedge [u_1]\wedge \cdots \wedge \hat{[u_k]}\wedge \cdots \wedge [u_{p-1}] \\
                &=& \sum _{1\leq i<j\leq p-1}(-1)^{j-1}\langle u_i,u_j\rangle [u_1]\wedge \cdots \wedge \hat{[u_j]}\wedge \cdots \wedge [u_{p-1}]\wedge [u_p+u_j] \\
                &&+\sum _{1\leq i<j\leq p-1}(-1)^{i}\langle u_i,u_j\rangle [u_1]\wedge \cdots \wedge \hat{[u_i]}\wedge \cdots \wedge [u_{p-1}]\wedge [u_p+u_i] \\
                &&+\sum _{1\leq k<\ell \leq p-1}(-1)^{k-1}\langle u_k,u_{\ell} \rangle [u_1]\wedge \cdots \wedge \hat{[u_k]}\wedge \cdots \wedge [u_{p-1}]\wedge [u_p+u_k] \\
                &&+\sum _{1\leq \ell <k \leq p-1}(-1)^{k}\langle u_{\ell},u_k \rangle [u_1]\wedge \cdots \wedge \hat{[u_k]}\wedge \cdots \wedge [u_{p-1}]\wedge [u_p+u_k] \\
                &=&0 \qquad \mbox{mod }\hat{I}_{p,(z)}. 
            \end{eqnarray*}
            This completes the proof of the proposition. 
        \end{proof}
        By Proposition \ref{comp zero}, we can define a $\mbQ$-linear map $H_p(\mbQ [H])_{(z)}\to \hat{C}_p(\mbQ [H])_{(z)}$ and $\hat{C}^p(\mbQ [H])_{(z)}\to H^p(\mbQ [H])_{(z)}$. 
        \begin{Prop} \label{QotimesH/Zz}
            There is a natural $\mbQ$-linear isomorphism 
            \[ \hat{C}_p(\mbQ [H])_{(z)}\longleftarrow \mbQ \otimes _{\mbZ} (\wedge ^{p-1}_{\mbZ}H/\mbZ z). \]
        \end{Prop}
        \begin{proof}
        Define a map $f: \{ (u_1,\cdots ,u_p)\mid u_i\in H, u_1+\cdots +u_p=z\} \to \mbQ \otimes (\wedge ^{p-1}H/\mbZ z)$ by $f(u_1,\cdots ,u_p)=1\otimes (u_1\wedge \cdots \wedge u_{p-1})$. 
        We can consider $\mbQ$-$p$-linear extension $f([u_1],\cdots ,[u_p])=f(u_1,\cdots ,u_p)$. 
        We denote by $S_p$ the symmetric group of degree $p$. 
        For $\sigma \in S_p$ with $\sigma (p)=p$, we have $f(u_{\sigma (1)},\cdots ,u_{\sigma (p)})=({\rm sgn}\sigma )f(u_1,\cdots ,u_p)$ easily. 
        For $u_1,\ldots ,u_p\in H$ with $u_1+\cdots u_p=z$, we have 
        \begin{eqnarray*}
            f(u_p,u_2,\ldots ,u_{p-1},u_1)
            &=&1\otimes (u_p\wedge u_2\wedge \cdots \wedge u_{p-1}) \\
            &=& 1\otimes (z\wedge u_2\wedge \cdots \wedge u_{p-1}) -\sum _{i=1}^{p-1}1\otimes (u_i\wedge \cdots \wedge u_{p-1}) \\
            &=& -1\otimes (u_1\wedge \cdots \wedge u_{p-1}) \\
            &=& -f(u_1,\cdots ,u_p). 
        \end{eqnarray*}
        Hence we have $f(u_{\sigma (1)},\ldots ,u_{\sigma (p)})=({\rm sgn}\sigma )f(u_1,\ldots ,u_p)$ for $\sigma \in S_p$. 
        This induces the $\mbQ$-linear map $f:C_p(\mbQ [H])_{(z)}\to \mbQ \otimes (\wedge ^{p-1}H/\mbZ z)$ with $f([u_1]\wedge \cdots \wedge [u_p])=1\otimes (u_1\wedge \cdots \wedge u_{p-1})$. 
        For $u_1,v_1,u_2,\ldots ,u_p\in H$ with $u_1+v_1+u_2+\cdots +u_p=z$, we have 
        \begin{eqnarray*}
            &&f(u_1+v_1,u_2,\ldots ,u_p) \\
            &=& 1\otimes ((u_1+v_1)\wedge u_2\wedge \cdots \wedge u_{p-1}) \\
            &=& 1\otimes (u_1\wedge u_2\wedge \cdots \wedge u_{p-1})+1\otimes (v_1\wedge u_2\wedge \cdots \wedge u_{p-1}) \\
            &=& f(u_1,u_2,\ldots ,u_{p-1},u_p+v_1)+ f(v_1,u_2,\ldots ,u_{p-1},u_p+u_1). 
        \end{eqnarray*}
        This mean $f(I_{p,(z)})=0$ and $f$ induce $f:\hat{C}_p(\mbQ [H])_{(z)}\to \mbQ \otimes (\wedge ^{p-1}H/\mbZ z)$ with $f([u_1]\wedge \cdots \wedge [u_p])=1\otimes (u_1\wedge \cdots \wedge u_{p-1})$. 
        
        We construct the inverse map. 
        Define the map $g:\mbQ \times H\times \cdots \times H\to \hat{C}_p(\mbQ [H])_{(z)}$ by $g(c,u_1,\cdots ,u_{p-1})=c[u_1]\wedge \cdots \wedge [u_{p-1}]\wedge [z-u_1-\cdots -u_{p-1}]$. 
        It is trivial that $g$ satisfies $g(a+b,u_1,\ldots ,u_{p-1})=g(a,u_1,\ldots ,u_{p-1})+g(b,u_1,\ldots ,u_{p-1})$ and $g(c,u_{\sigma (1)},\ldots ,u_{\sigma (p-1)})=({\rm sgn}\sigma ) g(c,u_1,\ldots ,u_{p-1})$ for $\sigma \in S_{p-1}$. 
        Moreover the map $g$ is also $\mbZ$-$p$-linear. 
        In fact, we have 
        \begin{eqnarray*}
            &&g(c,u_1+v_1,u_2,\ldots ,u_{p-1}) \\
            &=& c[u_1+v_1]\wedge \cdots \wedge [u_{p-1}]\wedge [z-u_1-v_1-u_2-\cdots -u_{p-1}] \\
            &\equiv & c[u_1]\wedge \cdots \wedge [u_{p-1}]\wedge [z-u_1-\cdots -u_{p-1}] \\
            &&\quad +c[v_1]\wedge \cdots \wedge [u_{p-1}]\wedge [z-v_1-u_2-\cdots -u_{p-1}] \\
            &=& g(c,u_1,\ldots ,u_{p-1})+g(c,v_1,u_2,\ldots ,u_{p-1}). 
        \end{eqnarray*}
        Moreover we have 
        \begin{eqnarray*}
            g(c,z,u_2,\ldots ,u_{p-1})
            &=& c[z]\wedge [u_2]\wedge \cdots \wedge [u_{p-1}]\wedge [z-z-u_2-\cdots -u_{p-1}] \\
            &=& -\sum _{i=2}^{p-1} c[z+u_i]\wedge [u_2]\wedge \cdots \wedge [u_{p-1}]\wedge [u_i] \\
            &=&0. 
        \end{eqnarray*}
        Hence the map $g$ induces the $\mbQ$-linear $g:\mbQ \otimes (\wedge ^{p-1}H/\mbZ z)\to \hat{C}_p(\mbQ [H])_{(z)}$ with $g(c\otimes (u_1\wedge \cdots \wedge u_{p-1}))=c[u_1]\wedge \cdots \wedge [u_{p-1}]\wedge [z-u_1-\cdots -u_{p-1}]$. 
        
        Clearly $g$ is the inverse of $f$. 
        \end{proof}
        Dually, we have an natural isomorphism $C^p(\mbQ [H])_{(z)}\leftarrow {\rm Hom}_{\mbZ}(\wedge ^2H/\mbZ z,\mbQ )$. 
        For the rest of this subsection, we confine ourselves to the only second homology group. 
        We have $C_2(\mbQ [H^{(1)}])_{(z)}=Z_2(\mbQ [H^{(1)}])_{(z)}$ and $C_2(\mbQ [H])_{(z)}=Z_2(\mbQ [H])_{(z)}$. 
        In fact, $\partial _2([u]\wedge [z-u])=-\langle u,z-u\rangle [z]=0$ for $u\in H^{(1)}$ or $u\in H$. 
        Set the inclusion $\iota ^{(1)}:\mbQ [H^{(1)}]\to \mbQ [H]$. 
        \begin{Prop}\label{inj}
            The kernel $C_2(\iota ^{(1)})^{-1}(I_{2,(z)})$ of the composition $Z_2(\mbQ [H^{(1)}])_{(z)} \to Z_2(\mbQ [H])_{(z)}\to \hat{C}_2(\mbQ [H])_{(z)}$ equals $B_2(\mbQ [H^{(1)}])_{(z)}$. 
        \end{Prop}
        \begin{proof}
            For $u,v\in H^{(1)}$ with $u+v\in H^{(1)}$, we have 
            \begin{eqnarray*}
                &&C_2(\iota ^{(1)}) \circ \partial _3([u]\wedge [v]\wedge [z-u-v]) \\
                &=& -\langle u,v\rangle \left( [u+v]\wedge [z-u-v]-[u]\wedge [z-u]-[v]\wedge [z-v]\right) \\
                &\equiv & 0 . 
            \end{eqnarray*}
            Hence we obtain $C_2(\iota ^{(1)})^{-1}(I_{2,(z)}) \supset B_2(\mbQ [H^{(1)}])_{(z)}$. 
            By similar calculation, we have $I_{2,(z)}\supset B_2(\mbQ [H])_{(z)}$. 
            
            $C_2(\iota ^{(1)})^{-1}(I_{2,(z)})$ is generated by the set $\{ [u+v]\wedge [z-u-v]-[u]\wedge [z-u]-[v]\wedge [z-v] \mid u,v,u+v\in H^{(1)}\}$. 
            If $\langle u,v\rangle \neq 0$, then we have $B_2(\mbQ [H^{(1)}])_{(z)}\ni \partial _3(\frac{-1}{\langle u,v\rangle}[u]\wedge [v]\wedge [z-u-v])=[u+v]\wedge [z-u-v]-[u]\wedge [z-u]-[v]\wedge [z-v]$. 
            If $u,v,u+v\in H^{(1)}$ and $\langle u,v\rangle =0$, we can take $x\in H^{(1)}$ with $\langle u,x\rangle \neq 0$, $\langle v,x\rangle \neq 0$ and $\langle u+v,x\rangle \neq 0$. 
            Then we obtain $[u+v]\wedge [z-u-v]-[u]\wedge [z-u]-[v]\wedge [z-v] \in B_2(\mbQ [H^{(1)}])_{(z)}$ since $-[u+v+x]\wedge [z-u-v-x]+[u+v]\wedge [z-u-v]+[x]\wedge [z-x] \in B_2(\mbQ [H^{(1)}])_{(z)}$, $[u+v+x]\wedge [z-u-v-x]-[u]\wedge [z-u]+[x+v]\wedge [z-x-v] \in B_2(\mbQ [H^{(1)}])_{(z)}$ and $[v+x]\wedge [z-v-x]-[v]\wedge [z-v]+[x]\wedge [z-x] \in B_2(\mbQ [H^{(1)}])_{(z)}$. 
            Hence we have $C_2(\iota ^{(1)})^{-1}(I_{2,(z)}) \subset B_2(\mbQ [H^{(1)}])_{(z)}$. 
            This completes the proof of the proposition. 
        \end{proof}
        By Proposition \ref{inj}, the composition $H_2(\mbQ [H^{(1)}])_{(z)}\to H_2(\mbQ [H])_{(z)} \to \hat{C}_2(\mbQ [H])_{(z)} \to \mbQ \otimes H/\mbZ z$ is injective. 
        \begin{Prop} \label{surj}
            The composition $Z_2(\mbQ [H^{(1)}])_{(z)}\to Z_2(\mbQ [H])_{(z)}\to \hat{C}_1(\mbQ [H])_{(z)}\to \mbQ \otimes H/\mbZ z$ is surjective if $\langle -,-\rangle \neq 0$. 
        \end{Prop}
        \begin{proof}
        We have $H^{(1)}\neq \emptyset$ since $\langle -,-\rangle \neq 0$. 
        Take $x_0\in H^{(1)}$. 
        The $\mbQ$-vector space $\mbQ \otimes H/\mbZ z$ is generated by the set $\{ 1\otimes u \mid u\in H\}$. 
        The set $\{ 1\otimes u \mid u\in H^{(1)}\}$  also generates $\mbQ \otimes H/\mbZ z$ since we have $1\otimes u=1\otimes (u-x_0)+1\otimes x_0$. 
        For $u\in H^{(1)}$, $[u]\wedge [z-u]\in Z^2(\mbQ [H^{(1)})_{(z)}$ corresponds to $1\otimes u$ by the composition. 
        \end{proof}
        By Propositions \ref{QotimesH/Zz}, \ref{inj} and \ref{surj}, we obtain a natural isomorphism $H_2(\mbQ [H^{(1)}])_{(z)}\to \mbQ \otimes (H/\mbZ z)$ if $H^{(1)}\neq \emptyset$. 
        We have an isomorphism $H_2(\mbQ [H])_{(z)}\cong H_2(\mbQ [{\rm ker}\mu ])_{(z)}\oplus H_2(\mbQ [H^{(1)}])_{(z)}$ by the decomposition $\mbQ [H]=\mbQ [{\rm ker}\mu ]\oplus \mbQ [H^{(1)}]$. 
            \section{Outer component}
        Fix an element $z\in H^{(1)}$ in this section. 
        The inclusion $\iota ^{(1)}:\mbQ [H^{(1)}]\to \mbQ [H]$ is a section of the projection $\varpi ^{(1)}:\mbQ [H]=\mbQ [{\rm ker}\mu ]\oplus \mbQ [H^{(1)}]\to \mbQ [H^{(1)}]$. 
        Then the homomorphism $H_p(\varpi ^{(1)})_{(z)}: H_p(\mbQ [H])_{(z)}\to H_p(\mbQ [H^{(1)}])_{(z)}$ is surjective. 
        So we have $H_p(\mbQ [H^{(1)}])_{(z)}=0$ if $H_p(\mbQ [H])_{(z)}=0$. 
        \begin{Prop}\label{outer}
            We have $H_2(\mbQ [H])_{(z)}=0$. 
        \end{Prop}
        \begin{proof}
            We can take $y\in H$ with $\langle y,z\rangle \neq 0$. 
            Define $\Phi _p:C_p(\mbQ [H])_{(z)}\to C_{p+1}(\mbQ [H])_{(z)}$, $p=1,2$, by $\Phi _1([z])=\frac{-1}{\langle y,z\rangle}[y]\wedge [z-y]$ and 
            \begin{eqnarray*}
                \Phi _2([u]\wedge [v])
                &=& \frac{-1}{\langle y,z\rangle}([y]\wedge [u-y]\wedge [v]+[y]\wedge [u]\wedge [v-y]) \\
                && +\frac{1}{2 \langle y,z\rangle}[2 y]\wedge [u-y]\wedge [v-y] \\
                && +\frac{\langle u-y,v-y\rangle}{2 \langle y,z\rangle ^2}[y]\wedge [2 y]\wedge [z-3 y] . 
            \end{eqnarray*}
            We can check $\Phi _1\circ \partial _2+\partial _3\circ \Phi _2=id _{C_2(\mbQ [H])_{(z)}}$. 
            This shows the proposition. 
        \end{proof}
        Combine this proposition and the results of the section \ref{sec:inner}, we obtain the main theorem. 
        If we consider the case when $\langle -,-\rangle$ is non-degenerate, we have the following corollary. 
        \begin{Cor}
        If $\langle -,-\rangle$ is non-degenerate, we have an isomorphism 
        \[
            H_2(\mbQ [H])\to H_2(\mbQ [H^{(1)}]) \to \mbQ \otimes H . 
        \]
        \end{Cor}
        We remark that this corollary holds also for the case $H=0$, that is, $H^{(1)}=\emptyset$. 
        By the universal coefficient theorem, we have another corollary. 
        \begin{Cor}
            If $\langle -,-\rangle \neq 0$, we have
            \begin{eqnarray*}
                H^2(\mbQ [H^{(1)}])_{(z)}&\cong & {\rm Hom}_{\mbZ}(H/\mbZ z,\mbQ ) \mbox{ , if } z\in \ker \mu , \\
                H^2(\mbQ [H^{(1)}])_{(z)}&=&0 \mbox{ , if } z\in H^{(1)}, \\
                H^2(\mbQ [H])_{(z)}&=& H^2(\mbQ [\ker \mu ])_{(z)}\oplus H^2(\mbQ [H^{(1)}])_{(z)} ,
            \end{eqnarray*}
            \[
                H^2(\mbQ [H^{(1)}])\cong \prod _{z\in \ker \mu}{\rm Hom}_{\mbZ}(H/\mbZ z,\mbQ ) \mbox{ and }
            \]
            \[ 
                H^2(\mbQ [H])={\rm Hom}_{\mbQ}(\wedge ^2\mbQ [\ker \mu ],\mbQ )\oplus H^2(\mbQ [H^{(1)}]). 
            \]
        \end{Cor}
\section{Applications}
Let $\mathfrak{g}_K$ be the kernel of the $\mbQ$-linear map $K:\mbQ [H]\to \mbQ \otimes H$, $[x]\mapsto 1\otimes x$. 
Then $\mathfrak{g}_K$ is a Lie subalgebra of $\mbQ [H]$. 
    \begin{Thm} \label{G_K 2nd hom}
        If $\langle -,-\rangle \neq 0$, then 
        the composition $H_2(\mathfrak{g}_K ) \to H_2(\mbQ [H]) \to H_2(\mbQ [H^{(1)}])$ is surjective, where the map $H_2(\mathfrak{g}_K )\to H_2(\mbQ [H])$ is the inclusion homomorphism and the map $H_2(\mbQ [H])\to H_2(\mbQ [H^{(1)}])$ is the induced map by the projection $\varpi ^{(1)}:\mbQ [H]=\mbQ [\ker \mu ]\oplus \mbQ [H^{(1)}]\to \mbQ [H^{(1)}]$. 
    \end{Thm}
    \begin{proof}
        By Theorem \ref{2nd homology thm}, $H_2(\mbQ [H^{(1)}])$ is generated by the set $\{ [u]\wedge [z-u]\mid z \in \ker \mu ,u\in H^{(1)}\}$ since $H^{(1)}\neq \emptyset$. 
    
        We consider the chain $([2 u]-2 [u] )\wedge ([u -2 u] -2[z-u]+[z]) \in C_2(\mathfrak{g}_K)$ for $z\in \ker \mu$ and $u\in H^{(1)}$. 
        This is a cycle since $\langle u,u\rangle =0$ and $\langle u,z\rangle =0$. 
        Applying the inclusion homomorphism to the cycle, then we have $([2 u]-2 [u] )\wedge ([z -2 u] -2[z-u]+[z]) \in Z_2(\mbQ [H])$. 
        Applying the map $H_2(\varpi ^{(1)})$ to the cycle, then we have $([2 u]-2 [u] )\wedge ([z -2 u] -2[z-u]) \in Z_2(\mbQ [H^{(1)}])$. 
        The homology class of the cycle equals the homology class of $6[u]\wedge [z-u]$. 
        In fact, we have 
        \begin{eqnarray*}
            &&([2 u]-2 [u] )\wedge ([z -2 u] -2[z-u]) \\ &=& [2 u]\wedge [z-2 u] - 2 [2 u]\wedge [z-u] -2 [u]\wedge [z-2 u] +4 [u]\wedge [z-u] . 
        \end{eqnarray*}
        By Theorem \ref{2nd homology thm}, we have $[u]\wedge [v]=0\in H_2(\mbQ [H^{(1)}])$ if $u+v\in H^{(1)}$ and $[2 u]\wedge [z-2 u]=2[u]\wedge [z-u]\in H_2(\mbQ [H^{(1)}])$. 
        This completes the proof of the proposition. 
    \end{proof}
    Let $\Sigma$ be a compact surface. 
    Let $g$ be the genus of $\Sigma$ and $r$ the number of the cardinality of the set of the connected components of the boundary of $\Sigma$. 
    We consider the surjection from the Goldman Lie algebra of $\Sigma$ onto the homological Goldman Lie algebra of the first homology group of $\Sigma$ with the intersection form. 
    
    We recall the definition of the homomorphism. 
    We identify $H_1(\Sigma ;\mbZ )=\pi _1(\Sigma ,*)^{\rm abel}=\pi _1(\Sigma ,*)/[\pi _1(\Sigma ,*),\pi _1(\Sigma ,*)]$. 
    Hence we have the abelianization map $q_a:\pi _1(\Sigma ,*)\to H_1(\Sigma ,\mbZ )$. 
    This is a group homomorphism. 
    We can identify $[S^1,\Sigma ]=\pi _1(\Sigma ,\mbZ )/{\rm conj.}$ since $\Sigma$ is connected. 
    Hence we have the quotient map $q_c:\pi _1(\Sigma ,\mbZ )\to [S^1,\Sigma ]$. 
    This map is given by forgetting the base point. 
    This induces the map between the second homology groups of the Lie algebras. 
    Take a projection $q:[S^1,\Sigma ]\to H_1(\Sigma ;\mbZ )$ with $q\circ q_c=q_a$. 
    This map is independent of the choice $*\in \Sigma$. 
    Set $q=\mbQ [q]:\mbQ [[S^1,\Sigma ]]\to \mbQ [H_1(\Sigma ;\mbZ )]$ the $\mbQ$-linear extension of $q$. 
    This is the desired Lie algebra homomorphism. 
    \begin{Thm} \label{Goldman 2nd hom}
        The composition $H_2(\varpi ^{(1)})\circ H_2(q):H_2(\mbQ [[S^1,\Sigma ]])\to H_2(\mbQ [H_1(\Sigma ;\mbZ )]) \to H_2(\mbQ [H_1(\Sigma ;\mbZ )^{(1)}])$ is surjective if $\Sigma$ is connected. 
    \end{Thm}
    \begin{proof}
        If $H^{(1)}=\emptyset$, then we have $\mbQ [H_1(\Sigma ;\mbZ )^{(1)}]=0$. 
        Hence the proposition holds. 
        Assume $H^{(1)}\neq \emptyset$. 
        Then we have $g\geq 1$. 
        Fix a base point $*\in \Sigma$. 
        Take based oriented loops $\alpha _1,\beta _1,\ldots ,\alpha _g, \beta _g, \gamma _1,\cdots ,\gamma _r$ as follows. 
        ($g=2$, $r=3$)
        
        \input{S_2,3_with_symplectic_basis.tex}

        The set $\{ \alpha _1,\beta _1,\ldots ,\alpha _g, \beta _g, \gamma _1,\cdots ,\gamma _r \}$ is a generator system of the fundamental group $\pi _1(\Sigma ,*)$ of $\Sigma$. 
        They satisfy the relation 
        \[
            \alpha _1\beta _1 \alpha _1^{-1}\beta _1^{-1}\cdots \alpha _g\beta _g \alpha _g^{-1}\beta _g^{-1}\gamma _1\cdots \gamma _r=1
        \]
        which is a defining relation of the group $\pi _1(\Sigma ,*)$. 
        Set $q_a(\alpha _i)=A_i$, $q_a(\beta _i)=B_i$ and $q_a(\gamma _j)=C_j$. 
        They satisfy $\langle A_i,A_j\rangle =\langle B_1,B_j\rangle =\langle C_i,x\rangle =0$ and $\langle A_i,B_i\rangle =\delta _{i,j}$. 
        In particular, the kernel $\ker \mu$ is generated by $C_1,\ldots ,C_r\in H$ with the defining relation $C_1+\cdots +C_r=0$. 
        Theorem \ref{2nd homology thm} says that the set $\{ [u]\wedge [z-u]\mid x \in \{ A_1,B_1,\ldots ,A_g,B_g,C_1,\ldots ,C_r\} ,z\in \ker \mu \}$ generates $H_2(\mbQ [H^{(1)}])$ as a $\mbQ$-vector space. 
        
        Let $z\in \ker \mu$. 
        Then there exist $m_1,\ldots ,m_r\in \mbZ$ with $z=m_1 C_1+\cdots +m_r C_r$. 

        Assume $u=A_i$. 
        Set $\alpha =\alpha _i$ and $\beta =\gamma _r^{m_r}\cdots \gamma _1^{m_1}\alpha _i^{-1}$. 
        We may assume $q_c(\alpha )\cap q_c(\beta )=\emptyset$, that is, there exist $A\in q_c(\alpha )$ and $B\in q_c(\beta )$ with $A(S^1)\cap B(S^1)=\emptyset$. 
        This shows that $[q_c(\alpha )]\wedge [q_c(\beta )]\in C_2(\mbQ [[S^1,\Sigma ]])$ is a cycle by the definition of the bracket in $\mbQ [[S^1,\Sigma ]]$. 
        We have 
        \begin{eqnarray*}
            C_2(\varpi ^{(1)})\circ C_2(q) ([q_c(\alpha )]\wedge [q_c(\beta )] )
            &=& C_2(\varpi ^{(1)})([q\circ q_c(\alpha )]\wedge [q\circ q_c(\beta )]) \\
            &=& C_2(\varpi ^{(1)})([A_i]\wedge [z-A_i]). 
        \end{eqnarray*}
        Hence we have the homology class of $[A_i]\wedge [z-A_i]$ is in the image of $H_2(\varpi ^{(1)})\circ H_2(q)$. 
        
        Assume $u=B_i$. 
        If we set $\alpha =\beta _i$ and $\beta =\gamma _r^{m_r}\cdots \gamma _1^{m_1}\beta _i^{-1}$, then we have that the homology class of $[B_i]\wedge [z-B_i]$ is included in the image of $H_2(\varpi ^{(1)})\circ H_2(q)$ similarly. 
        
        Assume $u=C_j$. 
        We remark $C_2(\varpi ^{(1)})\circ C_2(q)([q_c(\gamma _j)]\wedge [q_c(\gamma _r^{m_r}\cdots \gamma _1^{m_1}\gamma _j^{-1})])=0$ since $C_j\in \ker \mu$. 
        By Theorem \ref{2nd homology thm}, we have $[C_j]\wedge [z-C_j]=[C_j-A_g]\wedge [z-C_j+A_g]+[A_g]\wedge [z-A_g] \in H_2(\mbQ [H_1(\Sigma ;\mbZ )^{(1)}])$. 
        Hence it is enough to show that $[C_j-A_g]\wedge [z-C_j+A_g]$ is in the image. 
        Set $\alpha =\gamma _j \gamma _r^{-1}\cdots \gamma _r^{-1}\alpha _g^{-1}$ and 
        \[
            \beta =\alpha _g\gamma _1\cdots \gamma _r \gamma _r^{m_r-m_j}\cdots \gamma _{j+1}^{m_{j+1}-m_j}\gamma _j^{-1}\gamma _{j-1}^{m_{j-1}-m_j}\cdots \gamma _1^{m_1-m_j}. 
        \]
        We may also assume $q_c(\alpha )\cap q_c(\beta )=\emptyset$. 
        For example, see the following figure ($g=2$, $r=3$, $j=2$, $m_1-m_j=2$, $m_3-m_j=1$). 

        \input{S_2,3_with_arrows2.tex}
        
        We have $q_a(\alpha )=C_j-A_g$ and $q_a(\beta )=z-C_j+A_g$ by $C_1+\cdots +C_r=0$. 
        This completes the proof of the proposition. 
    \end{proof}
    Essentially, we have the following lemma at the proof of the Proposition \ref{inj}. 
    \begin{Lem}\label{linear extension}
        Let $f:H^{(1)}\to \mbZ$ be a function satisfying $f(u+v)=f(u)+f(v)$ for $u,v\in H^{(1)}$ with $\langle u,v\rangle \neq 0$. 
        Then we have $f(u+v)=f(u)+f(v)$ for $u,v\in H^{(1)}$ with $u+v\in H^{(1)}$, and $f(n u)=n f(u)$ for $u\in H^{(1)}$ and $n\in \mbZ \setminus \{ 0\}$. 
    \end{Lem}
    \begin{proof}
        Take $u,v\in H^{(1)}$ with $u+v\in H^{(1)}$. 
        Then there exists $x\in H^{(1)}$ with $\langle u,x\rangle \neq 0$, $\langle v,x\rangle \neq 0$ and $\langle u+v,x\rangle \neq 0$. 
        If $\langle u,v\rangle \neq 0$, we have $f(u+v)=f(u)+f(v)$ by the assumption of the lemma. 
        Assume $\langle u,v\rangle =0$. 
        Then we have $\langle u,v+x\rangle =\langle u,x\rangle \neq 0$. 
        By the assumption of the lemma, we have $f(u+v)=f(u)+f(v)$ because 
        \begin{eqnarray*}
            f(u+v)
            &=&f(u+v+x)-f(x) 
             = f(u)+f(v+x)-f(x) \\
            &=& f(u)+f(v)+f(x)-f(x)=f(u)+f(v) . 
        \end{eqnarray*}
        
        Take $u\in H^{(1)}$ and $n\in \mbZ \setminus \{ 0\}$. 
        Then there exists $x\in H^{(1)}$ with $\langle u,x\rangle \neq 0$. 
        By the result of the first half of the lemma, it is enough to show the case $n=-1$. 
        We obtain $f(-u)=-f(u)$ because 
        \begin{eqnarray*}
            f(-u)
            &=&-f(u+x)+f(x)
            = -f(u)-f(x)-f(x) 
            =-f(u). 
        \end{eqnarray*}
        This completes the proof of the lemma. 
    \end{proof}
    Assume $z\in \ker \mu$. 
    Then we have $\langle -,-\rangle \in {\rm Hom}_{\mbZ}(\wedge ^2 H/\mbZ z,\mbQ )$. 
    This corresponds the cocycle $\omega \in Z^3(\mbQ [H^{(1)}])_{(z)}$ or $\omega \in Z^3(\mbQ [H])_{(z)}$ with $\omega ([u],[v],[z-u-v])=\langle u,v\rangle$ for $u,v\in H^{(1)}$ by Proposition  \ref{QotimesH/Zz}. 
    The projection $\varpi ^{(1)}:\mbQ [H]=\mbQ [\ker \mu ]\oplus \mbQ [H^{(1)}]\to \mbQ [H^{(1)}]$ induces injection $H^3(\varpi ^{(1)})_{(z)}: H^3(\mbQ [H^{(1)}])_{(z)}\to H^3(\mbQ [H])_{(z)}$ since $\varpi ^{(1)} \circ \iota ^{(1)}=id _{\mbQ [H^{(1)}]}$, where the map $\iota ^{(1)}:\mbQ [H^{(1)}]\to \mbQ [H]$ is the inclusion. 
    Therefore $[\omega ]=0\in H^3(\mbQ [H])_{(z)}$ if $[\omega ]=0 \in H^3(\mbQ [H^{(1)}])_{(z)}$, and $[\omega ]\neq 0 \in H^3(\mbQ [H^{(1)}])_{(z)}$ if $[\omega ]\neq 0\in H^3(\mbQ [H])_{(z)}$. 
    \begin{Thm} \label{3rd coh}
        Assume $\langle -,-\rangle \neq 0$. 
        If $z\in \ker \mu$ is a torsion element, i.e., there exists $n\in \mbZ\setminus \{ 0\}$ with $nz=0$, then we have $[\omega ]\neq 0\in H^3(\mbQ [H])_{(z)}$. 
        If $z\in \ker \mu$ is not a torsion element, i.e., $nz\neq 0$ for any $n\in \mbZ \setminus \{ 0\}$, then we have $[\omega ]=0\in H^3(\mbQ [H^{(1)}])_{(z)}$. 
    \end{Thm}
    \begin{proof}
        {\bf Case 1: } 
        Assume that $z$ is a torsion element. 
        Then there exists $n\in \mbZ\setminus \{ 0\}$ with $nz=0$. 
        Assume $[\omega ]=0\in H^3(\mbQ [H])_{(z)}$. 
        Then there exists $\eta \in C^2(\mbQ [H])_{(z)}$ with $d\eta =\omega$. 
        Set the map $f:H\to \mbQ$ by $f(x)=\eta ([x],[z-x])-1$ for $x\in H$. 
        For $u,v\in H$ with $\langle u,v\rangle \neq 0$, we have $f(u+v)=f(u)+f(v)$ since 
        \begin{eqnarray*}
            \langle u,v\rangle 
            &=& \omega ([u],[v],[z-u-v]) \\
            &=& d\eta ([u],[v],[z-u-v]) \\
            &=& -\langle u,v\rangle (\eta ([u+v],[z-u-v])-\eta ([u],[z-u])-\eta ([v],[z-v])) . 
        \end{eqnarray*}
        Therefore we have a contradiction because 
        \begin{eqnarray*}
        n \eta ([z-u],[u])
        &=& n f(z-u)+n \\
        &=& f(n z-n u) +n \\
        &=& -n f(u) +n \\
        &=& -n \eta ([u],[z-u])+n +n \\
        &=& n \eta ([z-u],[u]) +2 n 
        \end{eqnarray*}
        for $u\in H^{(1)}$ by Lemma \ref{linear extension}. 
        Hence we obtain $[\omega ]\neq 0\in H^3(\mbQ [H])_{(z)}$. 
        
        {\bf Case 2: } 
        Assume that $z$ is not a torsion element. 
        Then there exists a $\mbZ$-linear map $f:H\to \mbQ$ with $f(z)=1$. 
        Set a map $\eta :\{ (u,v)\mid u,v\in H^{(1)} , u+v=z\} \to \mbQ$ by $\eta (u,z-u)=f(u)+1$. 
        This induces $\eta \in C^2(\mbQ [H^{(1)}])_{(z)}$ since $\eta (z-u,u)=f(z-u)-1=f(z)-f(u)-1=-\eta (u,z-u)$ for $u\in H^{(1)}$. 
        We have $\omega =d\eta$. 
        In fact, for $u,v\in H^{(1)}$ with $u+v\in H^{(1)}$, we have 
        \begin{eqnarray*}
            &&d\eta ([u],[v],[z-u-v]) \\
            &=& -\langle u,v\rangle (\eta ([u+v],[z-u-v])-\eta ([u],[z-u])-\eta ([v],[z-v])) \\
            &=& -\langle u,v\rangle (f(u+v)+1-f(u)-1-f(v)-1) \\
            &=& \langle u,v\rangle \\
            &=& \omega ([u],[v],[z-u-v]) . 
        \end{eqnarray*}
        Hence we obtain $[\omega ]=0 \in H^3(\mbQ [H^{(1)}])_{(z)}$. 
        This completes the proof of the proposition. 
    \end{proof}
    We have that the map ${\rm Hom}_{\mbZ}(\wedge ^2 H/\mbZ z,\mbQ )\to H^3(\mbQ [H^{(1)}])_{(z)}$ is not injective if $\ker \mu$ is not torsion free. 
    Moreover we have that the map ${\rm Hom}_{\mbZ}(\wedge ^2 H_1(T^2;\mbZ ),\mbQ )\to H^3(\mbQ [H_1(T^2;\mbZ )^{(1)}])_{(0)}$ is injective since the $\mbQ$-vector space ${\rm Hom}_{\mbZ}(\wedge ^2 H_1(T^2;\mbZ ),\mbQ )$ is generated by $\langle -,-\rangle$, where $T^2$ is the torus.

\noindent \textsc{Kazuki Toda\\
Graduate School of Mathematical Sciences, \\
University of Tokyo, \\
3-8-1 Komaba Meguro-ku, Tokyo 153-8914, JAPAN}\\
\noindent \texttt{E-mail address:ktoda@ms.u-tokyo.ac.jp}

\end{document}

%% file: S_2,3_with_symplectic_basis.tex
\unitlength 0.1in
\begin{picture}( 42.0000, 17.5000)(  2.0000,-20.0000)
%
{\color[named]{Black}{%
\special{pn 20}%
\special{ar 800 1200 200 400  0.0000000  6.2831853}%
}}%
%
{\color[named]{Black}{%
\special{pn 8}%
\special{ar 800 1200 400 600  1.5707963  6.2831853}%
}}%
%
{\color[named]{Black}{%
\special{pn 8}%
\special{ar 1400 1200 400 600  1.5707963  3.1415927}%
}}%
%
{\color[named]{Black}{%
\special{pn 8}%
\special{ar 1600 1200 400 600  1.5707963  3.1415927}%
}}%
%
{\color[named]{Black}{%
\special{pn 8}%
\special{ar 1800 1200 400 600  1.5707963  3.1415927}%
}}%
%
{\color[named]{Black}{%
\special{pn 8}%
\special{pa 800 1800}%
\special{pa 2000 1800}%
\special{fp}%
\special{pa 1400 1200}%
\special{pa 1400 800}%
\special{fp}%
}}%
%
{\color[named]{Black}{%
\special{pn 8}%
\special{ar 1200 800 200 400  4.7123890  6.2831853}%
}}%
%
{\color[named]{Black}{%
\special{pn 8}%
\special{pa 1000 800}%
\special{pa 1000 794}%
\special{fp}%
\special{pa 1002 766}%
\special{pa 1002 758}%
\special{fp}%
\special{pa 1004 732}%
\special{pa 1004 724}%
\special{fp}%
\special{pa 1008 698}%
\special{pa 1008 690}%
\special{fp}%
\special{pa 1012 664}%
\special{pa 1014 658}%
\special{fp}%
\special{pa 1020 630}%
\special{pa 1022 624}%
\special{fp}%
\special{pa 1028 598}%
\special{pa 1030 590}%
\special{fp}%
\special{pa 1040 564}%
\special{pa 1042 558}%
\special{fp}%
\special{pa 1052 532}%
\special{pa 1054 526}%
\special{fp}%
\special{pa 1068 502}%
\special{pa 1070 496}%
\special{fp}%
\special{pa 1086 474}%
\special{pa 1090 468}%
\special{fp}%
\special{pa 1106 448}%
\special{pa 1112 442}%
\special{fp}%
\special{pa 1132 424}%
\special{pa 1138 420}%
\special{fp}%
\special{pa 1162 408}%
\special{pa 1168 406}%
\special{fp}%
\special{pa 1194 400}%
\special{pa 1200 400}%
\special{fp}%
}}%
%
{\color[named]{Black}{%
\special{pn 20}%
\special{ar 2000 1200 200 400  0.0000000  6.2831853}%
}}%
%
{\color[named]{Black}{%
\special{pn 8}%
\special{ar 2000 1200 400 600  1.5707963  6.2831853}%
}}%
%
{\color[named]{Black}{%
\special{pn 8}%
\special{ar 2600 1200 400 600  1.5707963  3.1415927}%
}}%
%
{\color[named]{Black}{%
\special{pn 8}%
\special{ar 2800 1200 400 600  1.5707963  3.1415927}%
}}%
%
{\color[named]{Black}{%
\special{pn 8}%
\special{ar 3000 1200 400 600  1.5707963  3.1415927}%
}}%
%
{\color[named]{Black}{%
\special{pn 8}%
\special{pa 2000 1800}%
\special{pa 3200 1800}%
\special{fp}%
\special{pa 2600 1200}%
\special{pa 2600 800}%
\special{fp}%
}}%
%
{\color[named]{Black}{%
\special{pn 8}%
\special{ar 2400 800 200 400  4.7123890  6.2831853}%
}}%
%
{\color[named]{Black}{%
\special{pn 8}%
\special{pa 2200 800}%
\special{pa 2200 792}%
\special{fp}%
\special{pa 2202 764}%
\special{pa 2202 758}%
\special{fp}%
\special{pa 2204 730}%
\special{pa 2204 722}%
\special{fp}%
\special{pa 2208 694}%
\special{pa 2208 688}%
\special{fp}%
\special{pa 2214 660}%
\special{pa 2214 652}%
\special{fp}%
\special{pa 2220 626}%
\special{pa 2222 620}%
\special{fp}%
\special{pa 2230 594}%
\special{pa 2232 586}%
\special{fp}%
\special{pa 2240 562}%
\special{pa 2242 554}%
\special{fp}%
\special{pa 2254 530}%
\special{pa 2256 524}%
\special{fp}%
\special{pa 2268 500}%
\special{pa 2272 494}%
\special{fp}%
\special{pa 2286 472}%
\special{pa 2292 466}%
\special{fp}%
\special{pa 2308 446}%
\special{pa 2314 440}%
\special{fp}%
\special{pa 2334 424}%
\special{pa 2340 420}%
\special{fp}%
\special{pa 2362 408}%
\special{pa 2368 406}%
\special{fp}%
\special{pa 2392 400}%
\special{pa 2400 400}%
\special{fp}%
}}%
%
{\color[named]{Black}{%
\special{pn 8}%
\special{pa 600 600}%
\special{pa 800 600}%
\special{ip}%
\special{sh 1}%
\special{pa 800 600}%
\special{pa 734 580}%
\special{pa 748 600}%
\special{pa 734 620}%
\special{pa 800 600}%
\special{fp}%
\special{pa 1400 1200}%
\special{pa 1400 1000}%
\special{ip}%
\special{sh 1}%
\special{pa 1400 1000}%
\special{pa 1380 1068}%
\special{pa 1400 1054}%
\special{pa 1420 1068}%
\special{pa 1400 1000}%
\special{fp}%
}}%
%
{\color[named]{Black}{%
\special{pn 8}%
\special{pa 1800 600}%
\special{pa 2000 600}%
\special{ip}%
\special{sh 1}%
\special{pa 2000 600}%
\special{pa 1934 580}%
\special{pa 1948 600}%
\special{pa 1934 620}%
\special{pa 2000 600}%
\special{fp}%
\special{pa 2600 1200}%
\special{pa 2600 1000}%
\special{ip}%
\special{sh 1}%
\special{pa 2600 1000}%
\special{pa 2580 1068}%
\special{pa 2600 1054}%
\special{pa 2620 1068}%
\special{pa 2600 1000}%
\special{fp}%
}}%
%
{\color[named]{Black}{%
\special{pn 20}%
\special{ar 600 1200 400 800  1.5707963  4.7123890}%
}}%
%
{\color[named]{Black}{%
\special{pn 20}%
\special{pa 600 400}%
\special{pa 3200 400}%
\special{fp}%
\special{pa 3200 2000}%
\special{pa 600 2000}%
\special{fp}%
}}%
%
{\color[named]{Black}{%
\special{pn 20}%
\special{pa 3200 400}%
\special{pa 4200 400}%
\special{fp}%
\special{pa 4200 800}%
\special{pa 3600 800}%
\special{fp}%
}}%
%
{\color[named]{Black}{%
\special{pn 20}%
\special{ar 4200 600 200 200  0.0000000  6.2831853}%
}}%
%
{\color[named]{Black}{%
\special{pn 8}%
\special{pa 1000 1200}%
\special{pa 1000 800}%
\special{dt 0.045}%
}}%
%
{\color[named]{Black}{%
\special{pn 8}%
\special{pa 2200 1200}%
\special{pa 2200 800}%
\special{dt 0.045}%
}}%
%
{\color[named]{Black}{%
\special{pn 20}%
\special{pa 3600 1000}%
\special{pa 4200 1000}%
\special{fp}%
\special{pa 4200 1400}%
\special{pa 3600 1400}%
\special{fp}%
\special{pa 3600 1600}%
\special{pa 4200 1600}%
\special{fp}%
\special{pa 4200 2000}%
\special{pa 3200 2000}%
\special{fp}%
}}%
%
{\color[named]{Black}{%
\special{pn 20}%
\special{ar 4200 1800 200 200  0.0000000  6.2831853}%
}}%
%
{\color[named]{Black}{%
\special{pn 20}%
\special{ar 4200 1200 200 200  0.0000000  6.2831853}%
}}%
%
{\color[named]{Black}{%
\special{pn 20}%
\special{ar 3600 900 100 100  1.5707963  4.7123890}%
}}%
%
{\color[named]{Black}{%
\special{pn 20}%
\special{ar 3600 1500 100 100  1.5707963  4.7123890}%
}}%
%
{\color[named]{Black}{%
\special{pn 4}%
\special{sh 1}%
\special{ar 3200 1800 16 16 0  6.28318530717959E+0000}%
}}%
%
{\color[named]{Black}{%
\special{pn 8}%
\special{pa 3200 1800}%
\special{pa 3200 600}%
\special{fp}%
}}%
%
{\color[named]{Black}{%
\special{pn 8}%
\special{ar 3600 600 400 200  3.1415927  4.7123890}%
}}%
%
{\color[named]{Black}{%
\special{pn 8}%
\special{pa 3600 400}%
\special{pa 3608 400}%
\special{fp}%
\special{pa 3634 404}%
\special{pa 3640 404}%
\special{fp}%
\special{pa 3664 410}%
\special{pa 3670 412}%
\special{fp}%
\special{pa 3690 422}%
\special{pa 3696 424}%
\special{fp}%
\special{pa 3714 436}%
\special{pa 3718 440}%
\special{fp}%
\special{pa 3736 454}%
\special{pa 3740 458}%
\special{fp}%
\special{pa 3756 474}%
\special{pa 3758 478}%
\special{fp}%
\special{pa 3772 496}%
\special{pa 3774 502}%
\special{fp}%
\special{pa 3784 520}%
\special{pa 3786 526}%
\special{fp}%
\special{pa 3794 550}%
\special{pa 3796 556}%
\special{fp}%
\special{pa 3800 580}%
\special{pa 3800 586}%
\special{fp}%
\special{pa 3800 614}%
\special{pa 3800 622}%
\special{fp}%
\special{pa 3796 646}%
\special{pa 3794 652}%
\special{fp}%
\special{pa 3786 676}%
\special{pa 3784 680}%
\special{fp}%
\special{pa 3774 700}%
\special{pa 3772 704}%
\special{fp}%
\special{pa 3758 722}%
\special{pa 3756 728}%
\special{fp}%
\special{pa 3740 744}%
\special{pa 3736 748}%
\special{fp}%
\special{pa 3718 762}%
\special{pa 3714 764}%
\special{fp}%
\special{pa 3696 776}%
\special{pa 3690 780}%
\special{fp}%
\special{pa 3670 788}%
\special{pa 3664 790}%
\special{fp}%
\special{pa 3640 796}%
\special{pa 3634 798}%
\special{fp}%
\special{pa 3608 800}%
\special{pa 3600 800}%
\special{fp}%
}}%
%
{\color[named]{Black}{%
\special{pn 8}%
\special{ar 3600 1000 400 200  3.1415927  4.7123890}%
}}%
%
{\color[named]{Black}{%
\special{pn 8}%
\special{pa 3800 400}%
\special{pa 3800 600}%
\special{ip}%
\special{sh 1}%
\special{pa 3800 600}%
\special{pa 3820 534}%
\special{pa 3800 548}%
\special{pa 3780 534}%
\special{pa 3800 600}%
\special{fp}%
}}%
%
{\color[named]{Black}{%
\special{pn 8}%
\special{ar 3600 1600 400 200  3.1415927  4.7123890}%
}}%
%
{\color[named]{Black}{%
\special{pn 8}%
\special{pa 3600 1000}%
\special{pa 3608 1000}%
\special{fp}%
\special{pa 3634 1004}%
\special{pa 3640 1004}%
\special{fp}%
\special{pa 3664 1010}%
\special{pa 3670 1012}%
\special{fp}%
\special{pa 3690 1022}%
\special{pa 3696 1024}%
\special{fp}%
\special{pa 3714 1036}%
\special{pa 3718 1040}%
\special{fp}%
\special{pa 3736 1054}%
\special{pa 3740 1058}%
\special{fp}%
\special{pa 3756 1074}%
\special{pa 3758 1078}%
\special{fp}%
\special{pa 3772 1096}%
\special{pa 3774 1102}%
\special{fp}%
\special{pa 3784 1120}%
\special{pa 3786 1126}%
\special{fp}%
\special{pa 3794 1150}%
\special{pa 3796 1156}%
\special{fp}%
\special{pa 3800 1180}%
\special{pa 3800 1186}%
\special{fp}%
\special{pa 3800 1214}%
\special{pa 3800 1222}%
\special{fp}%
\special{pa 3796 1246}%
\special{pa 3794 1252}%
\special{fp}%
\special{pa 3786 1276}%
\special{pa 3784 1280}%
\special{fp}%
\special{pa 3774 1300}%
\special{pa 3772 1304}%
\special{fp}%
\special{pa 3758 1322}%
\special{pa 3756 1328}%
\special{fp}%
\special{pa 3740 1344}%
\special{pa 3736 1348}%
\special{fp}%
\special{pa 3718 1362}%
\special{pa 3714 1364}%
\special{fp}%
\special{pa 3696 1376}%
\special{pa 3690 1380}%
\special{fp}%
\special{pa 3670 1388}%
\special{pa 3664 1390}%
\special{fp}%
\special{pa 3640 1396}%
\special{pa 3634 1398}%
\special{fp}%
\special{pa 3608 1400}%
\special{pa 3600 1400}%
\special{fp}%
}}%
%
{\color[named]{Black}{%
\special{pn 8}%
\special{pa 3800 1000}%
\special{pa 3800 1200}%
\special{ip}%
\special{sh 1}%
\special{pa 3800 1200}%
\special{pa 3820 1134}%
\special{pa 3800 1148}%
\special{pa 3780 1134}%
\special{pa 3800 1200}%
\special{fp}%
}}%
%
{\color[named]{Black}{%
\special{pn 8}%
\special{ar 3600 1200 400 200  3.1415927  4.7123890}%
}}%
%
{\color[named]{Black}{%
\special{pn 8}%
\special{ar 3600 1800 400 200  3.1415927  4.7123890}%
}}%
%
{\color[named]{Black}{%
\special{pn 8}%
\special{pa 3800 1600}%
\special{pa 3800 1800}%
\special{ip}%
\special{sh 1}%
\special{pa 3800 1800}%
\special{pa 3820 1734}%
\special{pa 3800 1748}%
\special{pa 3780 1734}%
\special{pa 3800 1800}%
\special{fp}%
}}%
%
{\color[named]{Black}{%
\special{pn 8}%
\special{pa 3600 1600}%
\special{pa 3608 1600}%
\special{fp}%
\special{pa 3634 1604}%
\special{pa 3640 1604}%
\special{fp}%
\special{pa 3664 1610}%
\special{pa 3670 1612}%
\special{fp}%
\special{pa 3690 1622}%
\special{pa 3696 1624}%
\special{fp}%
\special{pa 3714 1636}%
\special{pa 3718 1640}%
\special{fp}%
\special{pa 3736 1654}%
\special{pa 3740 1658}%
\special{fp}%
\special{pa 3756 1674}%
\special{pa 3758 1678}%
\special{fp}%
\special{pa 3772 1696}%
\special{pa 3774 1702}%
\special{fp}%
\special{pa 3784 1720}%
\special{pa 3786 1726}%
\special{fp}%
\special{pa 3794 1750}%
\special{pa 3796 1756}%
\special{fp}%
\special{pa 3800 1780}%
\special{pa 3800 1786}%
\special{fp}%
\special{pa 3800 1814}%
\special{pa 3800 1822}%
\special{fp}%
\special{pa 3796 1846}%
\special{pa 3794 1852}%
\special{fp}%
\special{pa 3786 1876}%
\special{pa 3784 1880}%
\special{fp}%
\special{pa 3774 1900}%
\special{pa 3772 1904}%
\special{fp}%
\special{pa 3758 1922}%
\special{pa 3756 1928}%
\special{fp}%
\special{pa 3740 1944}%
\special{pa 3736 1948}%
\special{fp}%
\special{pa 3718 1962}%
\special{pa 3714 1964}%
\special{fp}%
\special{pa 3696 1976}%
\special{pa 3690 1980}%
\special{fp}%
\special{pa 3670 1988}%
\special{pa 3664 1990}%
\special{fp}%
\special{pa 3640 1996}%
\special{pa 3634 1998}%
\special{fp}%
\special{pa 3608 2000}%
\special{pa 3600 2000}%
\special{fp}%
}}%
%
{\color[named]{Black}{%
\special{pn 8}%
\special{ar 3600 1800 400 200  1.5707963  3.1415927}%
}}%
\put(6.0000,-5.8000){\makebox(0,0)[lb]{$\alpha _1$}}%
\put(10.0000,-3.8000){\makebox(0,0)[lb]{$\beta _1$}}%
\put(22.0000,-3.8000){\makebox(0,0)[lb]{$\beta _2$}}%
\put(18.0000,-5.8000){\makebox(0,0)[lb]{$\alpha _2$}}%
\put(34.0000,-5.8000){\makebox(0,0)[lb]{$\gamma _1$}}%
\put(34.0000,-11.8000){\makebox(0,0)[lb]{$\gamma _2$}}%
\put(34.0000,-17.8000){\makebox(0,0)[lb]{$\gamma _3$}}%
\end{picture}%

%% file: S_2,3_with_arrows2.tex
\unitlength 0.1in
\begin{picture}( 42.0000, 16.0000)(  2.0000,-20.0000)
%
{\color[named]{Black}{%
\special{pn 20}%
\special{ar 800 1200 200 400  0.0000000  6.2831853}%
}}%
%
{\color[named]{Black}{%
\special{pn 20}%
\special{ar 2000 1200 200 400  0.0000000  6.2831853}%
}}%
%
{\color[named]{Black}{%
\special{pn 20}%
\special{ar 600 1200 400 800  1.5707963  4.7123890}%
}}%
%
{\color[named]{Black}{%
\special{pn 20}%
\special{pa 600 400}%
\special{pa 3200 400}%
\special{fp}%
\special{pa 3200 2000}%
\special{pa 600 2000}%
\special{fp}%
}}%
%
{\color[named]{Black}{%
\special{pn 20}%
\special{pa 3200 400}%
\special{pa 4200 400}%
\special{fp}%
\special{pa 4200 800}%
\special{pa 3600 800}%
\special{fp}%
}}%
%
{\color[named]{Black}{%
\special{pn 20}%
\special{ar 4200 600 200 200  0.0000000  6.2831853}%
}}%
%
{\color[named]{Black}{%
\special{pn 20}%
\special{pa 3600 1000}%
\special{pa 4200 1000}%
\special{fp}%
\special{pa 4200 1400}%
\special{pa 3600 1400}%
\special{fp}%
\special{pa 3600 1600}%
\special{pa 4200 1600}%
\special{fp}%
\special{pa 4200 2000}%
\special{pa 3200 2000}%
\special{fp}%
}}%
%
{\color[named]{Black}{%
\special{pn 20}%
\special{ar 4200 1800 200 200  0.0000000  6.2831853}%
}}%
%
{\color[named]{Black}{%
\special{pn 20}%
\special{ar 4200 1200 200 200  0.0000000  6.2831853}%
}}%
%
{\color[named]{Black}{%
\special{pn 20}%
\special{ar 3600 900 100 100  1.5707963  4.7123890}%
}}%
%
{\color[named]{Black}{%
\special{pn 20}%
\special{ar 3600 1500 100 100  1.5707963  4.7123890}%
}}%
%
{\color[named]{Black}{%
\special{pn 8}%
\special{pa 3800 1000}%
\special{pa 3808 1000}%
\special{fp}%
\special{pa 3828 1008}%
\special{pa 3834 1012}%
\special{fp}%
\special{pa 3848 1026}%
\special{pa 3852 1030}%
\special{fp}%
\special{pa 3866 1048}%
\special{pa 3868 1054}%
\special{fp}%
\special{pa 3878 1074}%
\special{pa 3880 1080}%
\special{fp}%
\special{pa 3888 1102}%
\special{pa 3890 1108}%
\special{fp}%
\special{pa 3894 1132}%
\special{pa 3896 1138}%
\special{fp}%
\special{pa 3898 1164}%
\special{pa 3900 1172}%
\special{fp}%
\special{pa 3900 1198}%
\special{pa 3900 1206}%
\special{fp}%
\special{pa 3900 1234}%
\special{pa 3898 1240}%
\special{fp}%
\special{pa 3894 1266}%
\special{pa 3894 1274}%
\special{fp}%
\special{pa 3888 1296}%
\special{pa 3886 1304}%
\special{fp}%
\special{pa 3880 1324}%
\special{pa 3876 1330}%
\special{fp}%
\special{pa 3868 1350}%
\special{pa 3864 1354}%
\special{fp}%
\special{pa 3852 1372}%
\special{pa 3848 1376}%
\special{fp}%
\special{pa 3832 1390}%
\special{pa 3828 1392}%
\special{fp}%
\special{pa 3808 1400}%
\special{pa 3800 1400}%
\special{fp}%
}}%
%
{\color[named]{Black}{%
\special{pn 8}%
\special{ar 3800 1200 100 200  3.6052403  4.7123890}%
}}%
%
{\color[named]{Black}{%
\special{pn 8}%
\special{ar 3800 1200 100 200  1.5707963  2.3561945}%
}}%
%
{\color[named]{Black}{%
\special{pn 8}%
\special{pa 2200 400}%
\special{pa 2208 402}%
\special{fp}%
\special{pa 2230 410}%
\special{pa 2236 414}%
\special{fp}%
\special{pa 2254 430}%
\special{pa 2260 436}%
\special{fp}%
\special{pa 2274 458}%
\special{pa 2278 464}%
\special{fp}%
\special{pa 2290 486}%
\special{pa 2294 492}%
\special{fp}%
\special{pa 2304 516}%
\special{pa 2306 524}%
\special{fp}%
\special{pa 2316 548}%
\special{pa 2318 554}%
\special{fp}%
\special{pa 2326 580}%
\special{pa 2328 586}%
\special{fp}%
\special{pa 2336 612}%
\special{pa 2338 618}%
\special{fp}%
\special{pa 2344 644}%
\special{pa 2346 652}%
\special{fp}%
\special{pa 2352 678}%
\special{pa 2354 684}%
\special{fp}%
\special{pa 2358 710}%
\special{pa 2360 718}%
\special{fp}%
\special{pa 2364 744}%
\special{pa 2366 750}%
\special{fp}%
\special{pa 2370 776}%
\special{pa 2372 784}%
\special{fp}%
\special{pa 2374 810}%
\special{pa 2376 816}%
\special{fp}%
\special{pa 2380 842}%
\special{pa 2380 850}%
\special{fp}%
\special{pa 2384 876}%
\special{pa 2384 884}%
\special{fp}%
\special{pa 2386 910}%
\special{pa 2388 918}%
\special{fp}%
\special{pa 2390 944}%
\special{pa 2390 950}%
\special{fp}%
\special{pa 2392 978}%
\special{pa 2392 984}%
\special{fp}%
\special{pa 2394 1012}%
\special{pa 2396 1018}%
\special{fp}%
\special{pa 2396 1044}%
\special{pa 2396 1052}%
\special{fp}%
\special{pa 2398 1078}%
\special{pa 2398 1086}%
\special{fp}%
\special{pa 2400 1112}%
\special{pa 2400 1120}%
\special{fp}%
\special{pa 2400 1146}%
\special{pa 2400 1154}%
\special{fp}%
\special{pa 2400 1180}%
\special{pa 2400 1188}%
\special{fp}%
\special{pa 2400 1214}%
\special{pa 2400 1222}%
\special{fp}%
\special{pa 2400 1248}%
\special{pa 2400 1256}%
\special{fp}%
\special{pa 2400 1282}%
\special{pa 2400 1290}%
\special{fp}%
\special{pa 2398 1316}%
\special{pa 2398 1324}%
\special{fp}%
\special{pa 2396 1350}%
\special{pa 2396 1358}%
\special{fp}%
\special{pa 2396 1384}%
\special{pa 2394 1392}%
\special{fp}%
\special{pa 2394 1418}%
\special{pa 2392 1426}%
\special{fp}%
\special{pa 2390 1452}%
\special{pa 2390 1458}%
\special{fp}%
\special{pa 2388 1486}%
\special{pa 2386 1492}%
\special{fp}%
\special{pa 2384 1520}%
\special{pa 2384 1526}%
\special{fp}%
\special{pa 2380 1552}%
\special{pa 2380 1560}%
\special{fp}%
\special{pa 2376 1586}%
\special{pa 2374 1592}%
\special{fp}%
\special{pa 2370 1620}%
\special{pa 2370 1626}%
\special{fp}%
\special{pa 2366 1652}%
\special{pa 2364 1660}%
\special{fp}%
\special{pa 2360 1686}%
\special{pa 2358 1692}%
\special{fp}%
\special{pa 2354 1718}%
\special{pa 2352 1726}%
\special{fp}%
\special{pa 2346 1752}%
\special{pa 2344 1758}%
\special{fp}%
\special{pa 2338 1784}%
\special{pa 2336 1790}%
\special{fp}%
\special{pa 2328 1816}%
\special{pa 2326 1822}%
\special{fp}%
\special{pa 2318 1848}%
\special{pa 2316 1854}%
\special{fp}%
\special{pa 2306 1878}%
\special{pa 2304 1886}%
\special{fp}%
\special{pa 2294 1910}%
\special{pa 2290 1916}%
\special{fp}%
\special{pa 2278 1938}%
\special{pa 2274 1944}%
\special{fp}%
\special{pa 2258 1966}%
\special{pa 2254 1970}%
\special{fp}%
\special{pa 2236 1988}%
\special{pa 2230 1992}%
\special{fp}%
\special{pa 2208 2000}%
\special{pa 2200 2000}%
\special{fp}%
}}%
%
{\color[named]{Black}{%
\special{pn 8}%
\special{ar 2200 800 1000 400  3.1415927  4.7123890}%
}}%
%
{\color[named]{Black}{%
\special{pn 8}%
\special{pa 1200 800}%
\special{pa 1200 800}%
\special{fp}%
\special{pa 1200 800}%
\special{pa 1200 1600}%
\special{fp}%
}}%
%
{\color[named]{Black}{%
\special{pn 8}%
\special{ar 3200 2000 1000 200  3.1415927  4.7123890}%
}}%
%
{\color[named]{Black}{%
\special{pn 8}%
\special{pa 2600 400}%
\special{pa 2608 400}%
\special{fp}%
\special{pa 2630 410}%
\special{pa 2634 412}%
\special{fp}%
\special{pa 2652 428}%
\special{pa 2658 434}%
\special{fp}%
\special{pa 2672 454}%
\special{pa 2676 460}%
\special{fp}%
\special{pa 2688 482}%
\special{pa 2692 488}%
\special{fp}%
\special{pa 2702 512}%
\special{pa 2706 518}%
\special{fp}%
\special{pa 2714 542}%
\special{pa 2716 550}%
\special{fp}%
\special{pa 2726 574}%
\special{pa 2726 582}%
\special{fp}%
\special{pa 2734 606}%
\special{pa 2736 614}%
\special{fp}%
\special{pa 2742 640}%
\special{pa 2744 646}%
\special{fp}%
\special{pa 2750 672}%
\special{pa 2752 678}%
\special{fp}%
\special{pa 2758 704}%
\special{pa 2760 712}%
\special{fp}%
\special{pa 2764 738}%
\special{pa 2764 746}%
\special{fp}%
\special{pa 2770 772}%
\special{pa 2770 778}%
\special{fp}%
\special{pa 2774 804}%
\special{pa 2776 812}%
\special{fp}%
\special{pa 2778 838}%
\special{pa 2780 846}%
\special{fp}%
\special{pa 2784 872}%
\special{pa 2784 880}%
\special{fp}%
\special{pa 2786 906}%
\special{pa 2788 914}%
\special{fp}%
\special{pa 2790 940}%
\special{pa 2790 948}%
\special{fp}%
\special{pa 2792 974}%
\special{pa 2792 982}%
\special{fp}%
\special{pa 2794 1008}%
\special{pa 2796 1016}%
\special{fp}%
\special{pa 2796 1042}%
\special{pa 2796 1050}%
\special{fp}%
\special{pa 2798 1076}%
\special{pa 2798 1084}%
\special{fp}%
\special{pa 2800 1110}%
\special{pa 2800 1118}%
\special{fp}%
\special{pa 2800 1146}%
\special{pa 2800 1152}%
\special{fp}%
\special{pa 2800 1180}%
\special{pa 2800 1186}%
\special{fp}%
\special{pa 2800 1214}%
\special{pa 2800 1222}%
\special{fp}%
\special{pa 2800 1248}%
\special{pa 2800 1256}%
\special{fp}%
\special{pa 2800 1282}%
\special{pa 2800 1290}%
\special{fp}%
\special{pa 2798 1318}%
\special{pa 2798 1324}%
\special{fp}%
\special{pa 2796 1352}%
\special{pa 2796 1358}%
\special{fp}%
\special{pa 2796 1386}%
\special{pa 2794 1392}%
\special{fp}%
\special{pa 2792 1420}%
\special{pa 2792 1426}%
\special{fp}%
\special{pa 2790 1454}%
\special{pa 2790 1460}%
\special{fp}%
\special{pa 2788 1488}%
\special{pa 2786 1494}%
\special{fp}%
\special{pa 2784 1522}%
\special{pa 2782 1528}%
\special{fp}%
\special{pa 2780 1556}%
\special{pa 2780 1562}%
\special{fp}%
\special{pa 2776 1588}%
\special{pa 2774 1596}%
\special{fp}%
\special{pa 2770 1622}%
\special{pa 2770 1630}%
\special{fp}%
\special{pa 2764 1656}%
\special{pa 2764 1662}%
\special{fp}%
\special{pa 2758 1688}%
\special{pa 2758 1696}%
\special{fp}%
\special{pa 2752 1722}%
\special{pa 2750 1728}%
\special{fp}%
\special{pa 2744 1754}%
\special{pa 2742 1762}%
\special{fp}%
\special{pa 2736 1788}%
\special{pa 2734 1794}%
\special{fp}%
\special{pa 2726 1820}%
\special{pa 2726 1826}%
\special{fp}%
\special{pa 2716 1852}%
\special{pa 2714 1858}%
\special{fp}%
\special{pa 2704 1882}%
\special{pa 2702 1888}%
\special{fp}%
\special{pa 2692 1912}%
\special{pa 2688 1918}%
\special{fp}%
\special{pa 2676 1942}%
\special{pa 2672 1948}%
\special{fp}%
\special{pa 2658 1968}%
\special{pa 2652 1972}%
\special{fp}%
\special{pa 2634 1988}%
\special{pa 2630 1992}%
\special{fp}%
\special{pa 2608 2000}%
\special{pa 2600 2000}%
\special{fp}%
}}%
%
{\color[named]{Black}{%
\special{pn 8}%
\special{ar 2600 800 1000 400  3.1415927  4.7123890}%
}}%
%
{\color[named]{Black}{%
\special{pn 8}%
\special{pa 1600 800}%
\special{pa 1600 800}%
\special{fp}%
\special{pa 1600 800}%
\special{pa 1600 1600}%
\special{fp}%
}}%
%
{\color[named]{Black}{%
\special{pn 8}%
\special{ar 2100 1600 510 100  6.2831853  3.1415927}%
}}%
%
{\color[named]{Black}{%
\special{pn 8}%
\special{ar 3600 1600 1000 800  3.1415927  4.7123890}%
}}%
%
{\color[named]{Black}{%
\special{pn 8}%
\special{pa 3600 400}%
\special{pa 3608 400}%
\special{fp}%
\special{pa 3630 408}%
\special{pa 3634 412}%
\special{fp}%
\special{pa 3650 426}%
\special{pa 3654 430}%
\special{fp}%
\special{pa 3666 448}%
\special{pa 3668 454}%
\special{fp}%
\special{pa 3678 474}%
\special{pa 3680 480}%
\special{fp}%
\special{pa 3688 502}%
\special{pa 3688 508}%
\special{fp}%
\special{pa 3694 530}%
\special{pa 3696 536}%
\special{fp}%
\special{pa 3698 562}%
\special{pa 3700 570}%
\special{fp}%
\special{pa 3700 596}%
\special{pa 3700 604}%
\special{fp}%
\special{pa 3700 632}%
\special{pa 3698 638}%
\special{fp}%
\special{pa 3696 664}%
\special{pa 3694 672}%
\special{fp}%
\special{pa 3688 694}%
\special{pa 3688 700}%
\special{fp}%
\special{pa 3680 722}%
\special{pa 3678 728}%
\special{fp}%
\special{pa 3668 748}%
\special{pa 3666 754}%
\special{fp}%
\special{pa 3652 772}%
\special{pa 3648 776}%
\special{fp}%
\special{pa 3634 788}%
\special{pa 3630 792}%
\special{fp}%
\special{pa 3608 800}%
\special{pa 3600 800}%
\special{fp}%
}}%
%
{\color[named]{Black}{%
\special{pn 8}%
\special{pa 3800 400}%
\special{pa 3808 400}%
\special{fp}%
\special{pa 3828 408}%
\special{pa 3834 412}%
\special{fp}%
\special{pa 3848 426}%
\special{pa 3852 430}%
\special{fp}%
\special{pa 3866 448}%
\special{pa 3868 454}%
\special{fp}%
\special{pa 3878 474}%
\special{pa 3880 480}%
\special{fp}%
\special{pa 3888 502}%
\special{pa 3890 508}%
\special{fp}%
\special{pa 3894 532}%
\special{pa 3896 538}%
\special{fp}%
\special{pa 3898 564}%
\special{pa 3900 572}%
\special{fp}%
\special{pa 3900 598}%
\special{pa 3900 606}%
\special{fp}%
\special{pa 3900 634}%
\special{pa 3898 640}%
\special{fp}%
\special{pa 3894 666}%
\special{pa 3894 674}%
\special{fp}%
\special{pa 3888 696}%
\special{pa 3886 704}%
\special{fp}%
\special{pa 3880 724}%
\special{pa 3876 730}%
\special{fp}%
\special{pa 3868 750}%
\special{pa 3864 754}%
\special{fp}%
\special{pa 3852 772}%
\special{pa 3848 776}%
\special{fp}%
\special{pa 3832 790}%
\special{pa 3828 792}%
\special{fp}%
\special{pa 3808 800}%
\special{pa 3800 800}%
\special{fp}%
}}%
%
{\color[named]{Black}{%
\special{pn 8}%
\special{ar 3700 400 100 200  1.5707963  3.1415927}%
}}%
%
{\color[named]{Black}{%
\special{pn 8}%
\special{ar 3700 800 100 200  4.7123890  6.2831853}%
}}%
%
{\color[named]{Black}{%
\special{pn 8}%
\special{pa 1200 1000}%
\special{pa 1200 1200}%
\special{ip}%
\special{sh 1}%
\special{pa 1200 1200}%
\special{pa 1220 1134}%
\special{pa 1200 1148}%
\special{pa 1180 1134}%
\special{pa 1200 1200}%
\special{fp}%
\special{pa 1600 1400}%
\special{pa 1600 1200}%
\special{ip}%
\special{sh 1}%
\special{pa 1600 1200}%
\special{pa 1580 1268}%
\special{pa 1600 1254}%
\special{pa 1620 1268}%
\special{pa 1600 1200}%
\special{fp}%
}}%
%
{\color[named]{Black}{%
\special{pn 8}%
\special{ar 3400 400 400 200  6.2831853  1.5707963}%
}}%
%
{\color[named]{Black}{%
\special{pn 8}%
\special{ar 3600 600 200 400  1.5707963  3.1415927}%
}}%
%
{\color[named]{Black}{%
\special{pn 8}%
\special{pn 8}%
\special{pa 3570 1000}%
\special{pa 3578 1002}%
\special{fp}%
\special{pa 3602 1012}%
\special{pa 3608 1016}%
\special{fp}%
\special{pa 3626 1036}%
\special{pa 3630 1042}%
\special{fp}%
\special{pa 3644 1066}%
\special{pa 3646 1072}%
\special{fp}%
\special{pa 3656 1096}%
\special{pa 3658 1104}%
\special{fp}%
\special{pa 3664 1130}%
\special{pa 3666 1136}%
\special{fp}%
\special{pa 3670 1164}%
\special{pa 3670 1170}%
\special{fp}%
\special{pa 3670 1198}%
\special{pa 3670 1204}%
\special{fp}%
\special{pa 3670 1230}%
\special{pa 3670 1238}%
\special{fp}%
\special{pa 3666 1264}%
\special{pa 3664 1272}%
\special{fp}%
\special{pa 3658 1298}%
\special{pa 3656 1304}%
\special{fp}%
\special{pa 3646 1330}%
\special{pa 3644 1336}%
\special{fp}%
\special{pa 3630 1360}%
\special{pa 3626 1366}%
\special{fp}%
\special{pa 3608 1386}%
\special{pa 3602 1390}%
\special{fp}%
\special{pa 3578 1400}%
\special{pa 3570 1400}%
\special{fp}%
}}%
%
{\color[named]{Black}{%
\special{pn 8}%
\special{ar 3600 1800 200 400  3.1415927  4.7123890}%
}}%
%
{\color[named]{Black}{%
\special{pn 8}%
\special{ar 3600 2000 1000 100  3.1415927  4.7123890}%
}}%
%
{\color[named]{Black}{%
\special{pn 8}%
\special{pa 3800 1600}%
\special{pa 3808 1600}%
\special{fp}%
\special{pa 3828 1608}%
\special{pa 3834 1612}%
\special{fp}%
\special{pa 3848 1626}%
\special{pa 3852 1630}%
\special{fp}%
\special{pa 3866 1648}%
\special{pa 3868 1654}%
\special{fp}%
\special{pa 3878 1674}%
\special{pa 3880 1680}%
\special{fp}%
\special{pa 3888 1702}%
\special{pa 3890 1708}%
\special{fp}%
\special{pa 3894 1732}%
\special{pa 3896 1738}%
\special{fp}%
\special{pa 3898 1764}%
\special{pa 3900 1772}%
\special{fp}%
\special{pa 3900 1798}%
\special{pa 3900 1806}%
\special{fp}%
\special{pa 3900 1834}%
\special{pa 3898 1840}%
\special{fp}%
\special{pa 3894 1866}%
\special{pa 3894 1874}%
\special{fp}%
\special{pa 3888 1896}%
\special{pa 3886 1904}%
\special{fp}%
\special{pa 3880 1924}%
\special{pa 3876 1930}%
\special{fp}%
\special{pa 3868 1950}%
\special{pa 3864 1954}%
\special{fp}%
\special{pa 3852 1972}%
\special{pa 3848 1976}%
\special{fp}%
\special{pa 3832 1990}%
\special{pa 3828 1992}%
\special{fp}%
\special{pa 3808 2000}%
\special{pa 3800 2000}%
\special{fp}%
}}%
%
{\color[named]{Black}{%
\special{pn 8}%
\special{ar 2100 1600 900 200  6.2831853  3.1415927}%
}}%
%
{\color[named]{Black}{%
\special{pn 8}%
\special{ar 3700 1600 700 476  3.1415927  4.7266737}%
}}%
%
{\color[named]{Black}{%
\special{pn 8}%
\special{ar 3800 1900 200 300  3.1415927  4.7123890}%
}}%
%
{\color[named]{Black}{%
\special{pn 8}%
\special{ar 3400 2000 400 200  4.7123890  6.2831853}%
}}%
%
{\color[named]{Black}{%
\special{pn 8}%
\special{ar 3742 1600 542 294  3.1415927  4.7123890}%
}}%
%
{\color[named]{Black}{%
\special{pn 8}%
\special{pa 3200 1800}%
\special{pa 3200 1600}%
\special{fp}%
}}%
%
{\color[named]{Black}{%
\special{pn 8}%
\special{pa 3100 1378}%
\special{pa 3240 1260}%
\special{ip}%
\special{sh 1}%
\special{pa 3240 1260}%
\special{pa 3176 1288}%
\special{pa 3198 1294}%
\special{pa 3202 1318}%
\special{pa 3240 1260}%
\special{fp}%
\special{pa 3440 1378}%
\special{pa 3368 1404}%
\special{ip}%
\special{sh 1}%
\special{pa 3368 1404}%
\special{pa 3436 1400}%
\special{pa 3418 1386}%
\special{pa 3424 1362}%
\special{pa 3368 1404}%
\special{fp}%
\special{pa 3900 1204}%
\special{pa 3900 1220}%
\special{ip}%
\special{sh 1}%
\special{pa 3900 1220}%
\special{pa 3920 1154}%
\special{pa 3900 1168}%
\special{pa 3880 1154}%
\special{pa 3900 1220}%
\special{fp}%
\special{pa 2400 1420}%
\special{pa 2400 1200}%
\special{ip}%
\special{sh 1}%
\special{pa 2400 1200}%
\special{pa 2380 1268}%
\special{pa 2400 1254}%
\special{pa 2420 1268}%
\special{pa 2400 1200}%
\special{fp}%
}}%
%
{\color[named]{Black}{%
\special{pn 8}%
\special{pa 3402 1810}%
\special{pa 3432 1630}%
\special{ip}%
\special{sh 1}%
\special{pa 3432 1630}%
\special{pa 3402 1692}%
\special{pa 3424 1684}%
\special{pa 3442 1700}%
\special{pa 3432 1630}%
\special{fp}%
\special{pa 3394 1830}%
\special{pa 3448 1562}%
\special{ip}%
\special{sh 1}%
\special{pa 3448 1562}%
\special{pa 3416 1622}%
\special{pa 3438 1614}%
\special{pa 3454 1630}%
\special{pa 3448 1562}%
\special{fp}%
}}%
%
{\color[named]{Black}{%
\special{pn 8}%
\special{pa 3458 824}%
\special{pa 3410 650}%
\special{ip}%
\special{sh 1}%
\special{pa 3410 650}%
\special{pa 3408 720}%
\special{pa 3424 700}%
\special{pa 3448 708}%
\special{pa 3410 650}%
\special{fp}%
\special{pa 3458 850}%
\special{pa 3420 722}%
\special{ip}%
\special{sh 1}%
\special{pa 3420 722}%
\special{pa 3420 792}%
\special{pa 3436 772}%
\special{pa 3458 780}%
\special{pa 3420 722}%
\special{fp}%
}}%
%
{\color[named]{Black}{%
\special{pn 8}%
\special{pa 1600 1372}%
\special{pa 1600 1152}%
\special{ip}%
\special{sh 1}%
\special{pa 1600 1152}%
\special{pa 1580 1218}%
\special{pa 1600 1204}%
\special{pa 1620 1218}%
\special{pa 1600 1152}%
\special{fp}%
}}%
\put(17.2900,-7.6000){\makebox(0,0)[lb]{$\pi _c(\beta )$}}%
\put(8.0000,-7.3000){\makebox(0,0)[lb]{$\pi _c(\alpha )$}}%
%
{\color[named]{Black}{%
\special{pn 8}%
\special{pa 2800 1000}%
\special{pa 2800 1400}%
\special{ip}%
\special{sh 1}%
\special{pa 2800 1400}%
\special{pa 2820 1334}%
\special{pa 2800 1348}%
\special{pa 2780 1334}%
\special{pa 2800 1400}%
\special{fp}%
\special{pa 2800 1200}%
\special{pa 2800 1470}%
\special{ip}%
\special{sh 1}%
\special{pa 2800 1470}%
\special{pa 2820 1402}%
\special{pa 2800 1416}%
\special{pa 2780 1402}%
\special{pa 2800 1470}%
\special{fp}%
}}%
%
{\color[named]{Black}{%
\special{pn 8}%
\special{pa 3676 1386}%
\special{pa 3676 1160}%
\special{ip}%
\special{sh 1}%
\special{pa 3676 1160}%
\special{pa 3656 1226}%
\special{pa 3676 1212}%
\special{pa 3696 1226}%
\special{pa 3676 1160}%
\special{fp}%
\special{pa 3676 1360}%
\special{pa 3676 1224}%
\special{ip}%
\special{sh 1}%
\special{pa 3676 1224}%
\special{pa 3656 1292}%
\special{pa 3676 1278}%
\special{pa 3696 1292}%
\special{pa 3676 1224}%
\special{fp}%
}}%
\end{picture}%